\documentclass[a4paper, 10pt]{article}
\usepackage[latin1]{inputenc}
\usepackage[english]{babel}

\usepackage{amsmath}
\usepackage{amsthm}
\usepackage{amssymb}
\usepackage{color}
\usepackage{bbm}

\theoremstyle{plain}
\newtheorem{theorem}{Theorem}
\newtheorem{proposition}[theorem]{Proposition}
\newtheorem{lemma}[theorem]{Lemma}
\newtheorem{corollary}[theorem]{Corollary}
\newtheorem{remark}[theorem]{Remark}

\newcommand{\R}{\mathbb{R}}

\newcommand{\mD}{\mathcal{D}}
\newcommand{\mL}{\mathcal{L}}
\newcommand{\mM}{\mathcal{M}}
\newcommand{\mR}{\mathcal{R}}
\newcommand{\mT}{\mathcal{T}}
\newcommand{\mV}{\mathcal{V}}
\newcommand{\dd}{\, \text{d}}

\usepackage{hyperref}
\usepackage{vmargin}
\setmarginsrb{3cm}{2.5cm}{3cm}{1.5cm}{0cm}{0cm}{0cm}{1.5cm}

\begin{document}

\title{Taylor Expansions of the Value Function Associated with a Bilinear Optimal Control
Problem}
\author{Tobias Breiten\footnote{Institute of Mathematics, University of Graz, Austria. E-mail: tobias.breiten@uni-graz.at} \quad
Karl Kunisch\footnote{Institute of Mathematics, University of Graz, Austria and RICAM Institute, Austrian Academy of Sciences, Linz, Austria. E-mail: karl.kunisch@uni-graz.at} \quad
Laurent Pfeiffer\footnote{Institute of Mathematics, University of Graz, Austria. E-mail: laurent.pfeiffer@uni-graz.at}
}

\maketitle

\begin{abstract}
A general bilinear optimal control problem subject to an infinite-dimensional state equation is considered.
Polynomial approximations of the associated value function are derived around the steady state by repeated formal differentiation of the Hamilton-Jacobi-Bellman equation.
The terms of the approximations are described by multilinear forms, which can be obtained as solutions to generalized Lyapunov equations with recursively defined right-hand sides.
They form the basis for defining a suboptimal feedback law. The approximation properties of this feedback law are investigated. An application to the optimal control of a Fokker-Planck equation is also provided.
\end{abstract}

{\em Keywords:}
Value function, Hamilton-Jacobi-Bellman equation, bilinear control systems, Riccati equation, generalized Lyapunov equations, Fokker-Planck equation.

{\em AMS Classification:}
49J20, 49N35, 93D05, 93D15.

\section{Introduction}

In this article, we consider the following bilinear optimal control problem:
\begin{align}
\inf_{u \in L^2(0,\infty)} & \mathcal{J}(u,y_0) := \frac{1}{2} \int_0^\infty \| y (t) \|_Y^2 \dd
t + \frac{\alpha}{2} \int_0^\infty u(t)^2 \dd t, \label{eqProblemIntro} \\
& \text{where: }
 \left\{	\begin{array} {l}
\frac{\dd }{\dd t}y(t)= Ay(t)+(Ny(t)+B)u(t), \quad \text{for } t>0 \label{eqEquationIntro} \\
y(0)=y_0.
\end{array} \right.
\end{align}
Here, $V \subset Y \subset V^*$ is a Gelfand triple of real Hilbert spaces, $y_0 \in Y$, $A\colon \mD(A)\subset Y \rightarrow Y$ is the infinitesimal generator of an analytic $C_0$-semigroup $e^{At}$ on $Y$, $N \in \mathcal{L}(V,Y)$, $B\in Y$, and $\alpha >0$. Additional assumptions on the system, in particular a stabilizability assumption, will be made in subsections \ref{subsec:Setting} and \ref{subsecGenLyapunov}. The goal pursued with problem \eqref{eqProblemIntro} is the stabilization of the dynamical system \eqref{eqEquationIntro} around the steady state 0 when a perturbation $y_0$ is applied.
We denote by $\mathcal{V}$ the associated value function: for $y_0 \in Y$, $\mathcal{V}(y_0)$ is the value of problem \eqref{eqProblemIntro} with initial condition $y_0$.

Rather than investigating this problem as a mathematical programming problem, which associates an optimal open-loop control with a given initial value $y_0$, we take the perspective of designing an optimal feedback law.
The design of an optimal feedback law is intimately related to the computation of the value function $\mathcal{V}$, which is in general a very difficult task, since $y$ takes values in an infinite-dimensional space. Even after discretization, the computation time needed for obtaining $\mathcal{V}$ usually increases exponentially with the dimension of the discretized state space, a phenomenon known as the curse of dimensionality. Nonetheless, the computation of a feedback law, rather than an open-loop control, is particularly relevant in the context of stabilization problems.

The goal of this article is to construct a Taylor approximation of the value function at the origin, and to derive from this approximation a feedback law which generates good open-loop controls for small values of $y_0$.
We begin by proving the existence of a sequence of multilinear forms $\mT_k \colon Y^k \rightarrow \R$ such that for any $p \geq 2$,
\begin{equation*}
\mathcal{V}_p(y):= \sum_{k=2}^p \frac{1}{k!} \mT_k(y,...,y)
\end{equation*}
is a polynomial approximation of order $p+1$ of the value function
$\mathcal{V}$ in the neighborhood of $0$, that is to say
\begin{equation} \label{eq:errorEstim}
\mathcal{V}(y) - \mathcal{V}_p(y)= \mathcal{O}( \| y \|_Y^{p+1}).
\end{equation}
The sequence $(\mathcal{T}_k)_{k \geq 2}$ is constructed  by induction. The bilinear mapping $\mT_2$ is the solution to an algebraic operator Riccati equation. For all $k \geq 3$, the mapping $\mT_k$ is the solution to the following generalized Lyapunov equation: for all $z_1,...,z_k \in \mathcal{D}(A)$,
\begin{equation}\label{eq:Ljapu}
\sum_{i=1}^k \mathcal{T}_k(z_1,...,z_{i-1},A_{\Pi} z_i, z_{i+1},...,z_k)=
\mathcal{R}_k(z_1,...,z_k),
\end{equation}
where the operator $A_\Pi$ generates an exponentially  stable semigroup on $Y$ and the right-hand side $\mathcal{R}_k$ is known and depends on $N$, $B$, $\mathcal{T}_2$,...,$\mathcal{T}_{k-1}$ in an explicit fashion.
The terminology \emph{generalized Lyapunov equations} is motivated by the fact 
that \eqref{eq:Ljapu} can be seen as a generalization of operator Lyapunov 
equations, which can typically be written as follows:
\begin{equation*}
\mathcal{T}(A_{\Pi}z_1,z_2) + \mathcal{T}(z_1,A_{\Pi} z_2)= 
\mathcal{R}(z_1,z_2).
\end{equation*}
To achieve this task and to present the resulting expressions in an convenient manner, we exploit the symmetry structure of the formal derivatives of $\mathcal{V}$.
From the approximation $\mathcal{V}_p$ of the value function $\mathcal{V}$, we derive the following feedback law:
\begin{equation*}
\mathbf{u}_p(y) = -
\frac{1}{\alpha} D \mathcal{V}_p(y)(Ny+B)
\end{equation*}
and analyse the associated closed-loop system:
\begin{equation} \label{eqClosedLoopIntro}
\frac{\dd }{\dd t}y(t)= Ay(t)+ (Ny(t)+B) \mathbf{u}(y(t)), \quad y(0)= y_0.
\end{equation}
We denote by $\mathbf{U}_p(y_0)$ the open-loop generated by $\mathbf{u}_p$ for a given initial condition $y_0$, that is to say, $\mathbf{U}_p(y_0;t)= \mathbf{u}_p(y(t))$, where $y(t)$ is the solution to \eqref{eqClosedLoopIntro}.
On top of \eqref{eq:errorEstim}, we prove that
\begin{equation} \label{eqEstimate1}
\mathcal{J}(\mathbf{U}_p(y_0),y_0) \leq \mathcal{V}(y_0) + \mathcal{O}(\| y_0 \|_Y^{p+1}).
\end{equation}
In other words, we prove that the open-loop controls generated by $\mathbf{u}_p$ are $\mathcal{O}(\| y_0 \|^{p+1})$-optimal. We also prove for all $y_0$ sufficiently small, there exists an optimal control $\bar{u}$ such that
\begin{equation} \label{eqEstimate2}
\| \mathbf{U}_p(y_0)-\bar{u} \|_{L^2(0,\infty)} = \mathcal{O} \big( \| y_0 \|_Y^{(p+1)/2} \big).
\end{equation}

In the finite-dimensional case, expansion techniques for Lyapunov functions or for the value function associated with nonlinear control problems have a long history, which dates back at least to \cite{Alb61}.
To the best of our knowledge, our article is the first one dealing with Taylor expansions of any order for infinite-dimensional systems. A sophisticated analysis is also required for proving the well-posedness of the closed-loop system associated with $\mathbf{u}_p$.
Moreover, the convergence rate analysis has apparently received little attention so far, especially concerning the rate of convergence of the suboptimal controls to the optimal ones. As far as we know, estimates \eqref{eqEstimate1} and \eqref{eqEstimate2} are new. In this respect, we are only aware of the analysis done in \cite{Gar72} for systems of the form: $\frac{\text{d}}{\text{dt}} y(t)= Ay(t) + \varepsilon \varphi(y(t)) + Bu(t)$.

Let us mention some additional related literature. In \cite{Alb61}, the author considers a general stabilization problem for a nonlinear system that can be expanded in a power series around the origin. It is shown that the optimal control can be characterized in terms of a convergent power series as well. In \cite{Luk69}, the expandability of the optimal control for nonlinear analytic and differentiable systems is analyzed in detail. As in the other works on this topic, an important assumption is the local stabilizability of the underlying system. Moreover, it is shown in \cite{Luk69} that the lowest order terms of the approximation are defined by the linearized dynamics. For nonlinear systems with linear controls, in \cite{Gar72,GarES92}, the degree of approximation of the truncated Taylor series to the optimal control is analyzed. In \cite{CebC84}, the formal power series approach is discussed for the particular case of bilinear control systems. The explicit structure of the terms up to the third order are given and shown to be unique for locally stabilizable systems. More recent developments, which are based on Taylor series expansions and their use for a numerical approximation of the value function can be found in \cite{NavK07} and \cite{AguK14}, as well as in the survey article \cite{KreAH13}. For a further detailed overview on obtaining optimal feedback controls including numerical experiments in the finite-dimensional case, we refer to the survey \cite{BeeTB00} and to the references therein.

In the infinite-dimensional case, we are only aware of the results from \cite{TheBR10}, where a third-order approximation for a stabilization problem of the Burgers equation with the control entering linearly is investigated theoretically and numerically. To the best of our knowledge, a more general analysis of Taylor approximations for infinite-dimensional control systems does not exist yet.

Our article is structured as follows. Section \ref{sectionDerivationLyapunov} is a preparatory section. We show that if $\mathcal{V}$ is Fr\'echet differentiable, then it is the solution to some HJB equation. In Theorem \ref{thm:DifferentiabilityImpliesLyapunov}, we further show that if $\mathcal{V}$ is $(p+1)$-times differentiable in the neighborhood of 0, then $D^p \mathcal{V}(0)$ is a solution to a generalized Lyapunov equation. This result motivates the construction of $\mathcal{V}_p$. Our main contributions start in section \ref{sectionWellPosednessLyapunov}. In this section, we rigorously define the sequence of multilinear forms $(\mathcal{T}_k)_{k \geq 2}$, the polynomial approximations $\mathcal{V}_p$, and the feedback laws $\mathbf{u}_p$. In section \ref{sectionClosedLoop}, we prove the well-posedness of the closed-loop system associated with $\mathbf{u}_p$, in the neighborhood of $0$. In section \ref{sectionOptimalOpenLoopControl}, we prove the existence of an optimal (open-loop) control  and investigate some of its regularity properties. Section \ref{sectionTaylorExpandability} contains our main results: in Theorem \ref{thm:subOptimality}, we prove the error estimates \eqref{eq:errorEstim} and \eqref{eqEstimate1}. Estimate \eqref{eqEstimate2} is proved in Theorem \ref{thm:errorEstimForUp}.

\section{Analytical preliminaries}

\subsection{State equation}
\label{subsec:Setting}

Throughout the article, $V\subset Y\subset V^*$ denotes a Gelfand triple of real Hilbert spaces, where the embedding of $V$ into $Y$ is dense and compact and where $V^*$ stands for the topological dual of $V$. Further, $a \colon V \times V\rightarrow \R$ denotes a bounded $V$-$Y$ bilinear form on $V \times V$, i.e.\@ there exist $\nu >0$ and $\lambda \in \R$, such that
\begin{equation}\label{ass:A1}\tag{A1}
a(v,v)\geq \nu \| v \|^2_V- \lambda \| v \|^2_Y \quad \text{for all}\; v\in V.
\end{equation}
Associated with $a$, there exists a unique closed linear operator $A$ in $Y$ characterized by $\mD(A)=\{v\in V\colon w \mapsto a(v,w)
\,\text{is}\,Y\text{-continuous}\}$ and by
$\langle A v,w \rangle_Y =-a(v,w)$, for all $v \in \mD(A)$ and $w\in V$, see e.g.\@
\cite[Part II, Chapter 1, Section 2.7]{Benetal07}. Moreover, $A$ has a uniquely
defined extension as bounded linear operator in $\mathcal{L}(V,V^*)$, which
will be denoted by the same symbol, see.\@ \cite[Section 2.2]{Tan79}. Further, we choose
$N\in \mathcal{L}(V,Y)$, and thus $N^*\in \mathcal{L}(Y,V^*)$. We assume
that the restrictions of $N$ and $N^*$ to $\mD(A)$ and to $V$, together with $N$
satisfy
\begin{equation}\label{ass:A2} \tag{A2}
N\in \mathcal{L}(V,Y)\cap \mathcal{L}(\mD(A),V) \quad  \text{and} \quad N^*\in
\mathcal{L}(V,Y),
\end{equation}
where $\mD(A)$ is endowed with the graph norm. Moreover, we assume that $B \in
Y$ and we choose $\alpha>0$. The inner product on $Y$ is denoted  by $\langle
\cdot,\cdot\rangle$ or $\langle \cdot,\cdot\rangle_Y$ and duality between $V$
and $V^*$ by $\langle \cdot,\cdot\rangle_{V,V^*}$.
We are now prepared to state the problem under consideration:
\begin{equation} \label{eqProblem} \tag{$P$}
	\inf_{u \in L^2(0,\infty)} \mathcal{J}(u,y_0):=  \frac{1}{2}
\int_0^\infty
	\| S(u,y_0;t)\|^2_Y \dd t + \frac{\alpha}{2}
\int_0^\infty u(t)^2 \dd t,
\end{equation}
where $S(u,y_0;\cdot)$ is the solution to
\begin{equation}\label{eq2.1}
\left\{	\begin{array} {l}
\frac{\dd}{\dd t}y(t)=Ay(t)+Ny(t)u(t)+Bu(t), \quad \text{for} \;t>0,\\[1.5ex]
y(0)=y_0.
\end{array} \right.
\end{equation}
Here, $S(u,y_0)$ is referred to as solution of \eqref{eq2.1} if for each $T>0$, it
lies in the space
\begin{equation*}
W(0,T)=\left\{y\in L^2 (0,T;V):\frac{\dd}{\dd t}y\in L^2 (0,T;V^*)\right\}.
\end{equation*}
We recall that $W(0,T)$ is continuously embedded in $C([0,T],Y)$ \cite[Theorem
3.1]{LioM72}.
Let us note that the origin is a steady state of the uncontrolled  system
\eqref{eq2.1}.
Associated with \eqref{eqProblem} and \eqref{eq2.1}, we define the value function on $Y$:
\begin{equation*}
\mathcal{V} (y_0) = \inf_{u \in L^2(0,\infty)} \mathcal{J}(u,y_0).
\end{equation*}
The following lemma  summarizes some properties of equation \eqref{eq2.1}. The
proof is quite standard and therefore deferred to the Appendix.

\begin{lemma} \label{lemma:RegEstim}
Assume that \eqref{ass:A1} and \eqref{ass:A2} hold. For all $u \in
L^2(0,\infty)$ and for all $y_0\in Y$, there exists a unique solution $y$
to \eqref{eq2.1} and a continuous function $c$ such that
\begin{equation}\label{eq:KK7}
\|y\|_{W(0,T)} \le c(T,\|y_0\|_Y, \|u\|_{L^2(0,T)}).
\end{equation}
Moreover, there exists a constant $C>0$ such that for all $T \geq 0$, for all $u \in
L^2(0,\infty)$ and for all $y_0$ and $\tilde{y}_0 \in Y$, we have
\begin{align}
\| y \|_{L^\infty(0,T;Y)}^2 \leq \ & \big( \|y_0\|_Y^2 + C \| u \|_{L^2(0,T)}^2
\big)\; e^{ C(T + \| u \|_{L^2(0,T)}) }, \label{eq:RegEstim1} \\
\| \tilde{y}-y \|_{L^\infty(0,T;Y)}^2 \leq \ & \|\tilde{y}_0 - y_0\|_Y^2\; e^{
 C(T + \| u \|_{L^2(0,T)}) }. \label{eq:RegEstim2}
\end{align}
If further $y$ lies in $L^2(0,\infty;Y)$, the constant $C$ is such that
\begin{align}
\| y \|_{L^\infty(0,\infty;Y)}^2 \leq  \ &
\Big( \| y_0 \|_Y^2 + C\big( \| y \|_{L^2(0,\infty;Y)}^2 + \| u
\|_{L^2(0,\infty)}^2 \big) \Big)
e^{ C \| u \|_{L^2(0,\infty)}^2}, \label{eq:RegEstimBis1}
\\
\| y \|_{L^2(0,\infty;V)}^2 \leq \ &
C \Big( \| y \|_{L^2(0,\infty;Y)}^2 + \big( \| y \|_{L^\infty(0,\infty;Y)}^2 +1
\big) \| u \|_{L^2(0,\infty)}^2 \Big), \label{eq:RegEstimBis2}
\\
\left\| \frac{\mathrm{d} y}{\mathrm{d} t} \right\|_{L^2(0,\infty;V^*)}^2 \leq \
& C \Big( \| y \|_{L^2(0,\infty;V)}^2 + \big( \| y \|_{L^\infty(0,\infty;Y)}^2
+1 \big) \| u \|_{L^2(0,\infty)}^2 \Big). \label{eq:RegEstimBis3}
\end{align}
Additionally, $\lim_{T\to \infty}  \| y(T) \|_Y = 0$.
\end{lemma}

\begin{proposition}\label{prop:kk1}
If problem \eqref{eqProblem} admits a feasible control (i.e.\@ a control $u \in L^2(0,\infty)$ such that $\mathcal{J}(u,y_0) < \infty$), then it has a solution.
\end{proposition}

The proof uses standard arguments and it is therefore given in the Appendix. Note that in Section 5, we construct a feedback law generating feasible controls (for small values of $\| y_0 \|_Y$).

\begin{remark}\label{KK1}
{\em We recall some additional properties of the operator $A$  generated by
$a$. First, it is well known that $A$ generates an analytic semigroup, see e.g.\@
\cite[Sections 3.6 and 5.4]{Tan79}, that we denote by $e^{At}$. Let us set $A_0=A-\lambda I$, if $\lambda >0$ and $A_0=A$ otherwise. Then $-A_0$ has a bounded inverse in $Y$, see
\cite[page 75]{Tan79}, and in particular it is maximal accretive, see
\cite[20]{Tan79}. We have $\mD(A_0)= \mD(A)$ and the fractional powers of
$-A_0$ are well-defined.
In particular,
$\mD((-A_0)^{\frac{1}{2}})=[\mD(-A_0),Y]_{\frac{1}{2}}:=(\mD(-A_0),Y)_
{2,\frac{1}{2}}$ the real interpolation space with indices 2 and $\frac{1}{2}$,
see \cite[Proposition 6.1, Part II, Chapter 1]{Benetal07}.}
\end{remark}
For the following regularity result, we require that
\begin{equation}\label{ass:A3}\tag{A3}
[\mD(-A_0),Y]_{\frac{1}{2}} = [\mD(-A_0^*),Y]_{\frac{1}{2}}=V.
\end{equation}

\begin{lemma}\label{le2.2}
Let \eqref{ass:A1}-\eqref{ass:A3} hold. Then, there exists a continuous function $c$ such that for all $T>0$, for all $y_0\in V$, and for all $u \in L^2(0,T;Y)$ the solution to \eqref{eq2.1} satisfies $y \in H^1(0,T;Y)\cap L^2(0,T;\mD(-A_0))$ and the following estimate holds:
\begin{equation}\label{eq:KK5}
\|y\|_{H^1(0,T;Y)\cap L^2(0,T;\mD(-A_0))} \le c(T,\|y_0\|_V,\|u\|_{L^2(0,T)}).
\end{equation}
\end{lemma}
\begin{proof}
Let $y$ denote the solution to \eqref{eq2.1} and define
$z=(-A_0)^{\frac{1}{2}}y$. Then, $z$ satisfies
\begin{equation}\label{eq:KK6}
\left\{	\begin{array} {l}
\frac{\dd}{\dd t}z(t)=Az(t)+\tilde Ny(t)u(t)+ \tilde Bu(t), \quad \text{for}
\;t>0, \\[1.5ex]
z(0)= (-A_0)^{\frac{1}{2}}y_0,
\end{array} \right.
\end{equation}
where $\tilde N=(-A_0)^{\frac{1}{2}}N  (-A_0)^{-\frac{1}{2}}$  and $\tilde B=
(-A_0)^{\frac{1}{2}}B$. Since $(-A_0)^{-\frac{1}{2}} \in
{\mathcal{L}}(V,\mD(-A_0)) $
 and $(-A_0)^{\frac{1}{2}} \in {\mathcal{L}} (V,Y)$, we have $\tilde N \in
{\mathcal{L}}(V,Y)$ and $\tilde B \in V^*$, where we use \eqref{ass:A3}. Now we
can apply \eqref{eq:KK7} of Lemma \ref{lemma:RegEstim}  to obtain that $z\in
H^1(0,T;V^*)\cap L^2(0,T;V)$ and this implies that $y \in H^1(0,T;Y)\cap
L^2(0,T;\mD(-A_0))$ and that \eqref{eq:KK5} holds.
\end{proof}

\begin{remark}\label{rem.kk2}
{\em For finite dimensional  systems with $V=Y=\R^n$, assumptions \eqref{ass:A1}, \eqref{ass:A2}, and \eqref{ass:A3} are trivially satisfied. In Section \ref{sec:FP}, we describe an infinite-dimensional control problem associated with a Fokker-Planck equation for which the general assumptions are satisfied.}
\end{remark}

\subsection{Notation for multilinear forms and differentiability properties}

We denote by $B_Y(\delta)$ the closed ball of $Y$ with radius $\delta$ and center 0.
For $k \geq 1$, we make use of the following norm:
\begin{equation} \label{eq:NormYk}
\| (y_1,...,y_k) \|_{Y^k}= \max_{i=1,...,k} \| y_i \|_Y.
\end{equation}
We denote by $B_{Y^k}(\delta)$ the closed ball of $Y^k$ with radius $\delta$ and center 0, for the norm $\| \cdot \|_{Y^k}$.
For $k \geq 1$, we say that $\mathcal{T}\colon Y^k \rightarrow \R$ is a bounded multilinear form if for all $i \in \{ 1,...,k \}$ and for all $z_1$,...,$z_{i-1}$,$z_{i+1}$,...,$z_k \in Y^{k-1}$, the mapping
$z
\in Y \mapsto \mathcal{T}(z_1,...,z_{i-1},z,z_{i+1},...,z_k)$ is linear and
\begin{equation} \label{eq:OperatorNormTensor}
\| \mathcal{T} \| := \sup_{y \in B_{Y^k}(1)} | \mathcal{T}(y) | < \infty.
\end{equation}
We denote by $\mathcal{M}(Y^k,\mathbb R)$ the set of bounded multilinear forms.
For all $\mathcal{T} \in \mathcal{M}(Y^k,\mathbb{R})$ and for all $(z_1,...,z_k) \in Y^k$, \begin{equation} \label{eq:OperatorNormTensor2}
|\mathcal{T}(z_1,...,z_k)| \leq \| \mathcal{T} \| \, \prod_{i=1}^k \| z_i
\|_Y.
\end{equation}
Bounded multilinear forms $\mathcal{T} \in \mathcal{M}(Y^k,\mathbb{R})$ are said to be symmetric if for all $z_1$,...,$z_k \in Y^k$ and for all permutations $\sigma$ of $\{ 1,...,k \}$,
\begin{equation*}
\mathcal{T}(z_{\sigma(1)},...,z_{\sigma(k)})= \mathcal{T}(z_1,...,z_k).
\end{equation*}
Given two multilinear forms $\mT_1\in \mM(Y^k,\R)$ and $\mT_2\in \mM(Y^\ell,R)$, we denote by $\mathcal{T}_1 \otimes \mathcal{T}_2$ the bounded multilinear mapping which is defined for all $(y_1,...,y_{k+ \ell}) \in Y^{k+\ell}$ by
\begin{equation*}
(\mT_1\otimes \mT_2)(y_1,\dots,  y_{k+\ell}) = \mT_1(y_1,...,y_k) \mT_2(y_{k+1},...,y_{k + \ell}).
\end{equation*}
For $y \in Y$, we denote
\begin{equation*}
y^{\otimes k}= (y,...,y) \in Y^k.
\end{equation*}

\begin{lemma} \label{lemma:continuousTensors1}
Let $\mT\colon Y^k \rightarrow \R$ be a multilinear form. Then, $\mT \in
\mM(Y^k,\R)$ if and only if it is continuous.
In this case, it is also Lipschitz continuous on bounded subsets of $Y^k$.  More
precisely, for all $M>0$, for all $y$ and $v \in B_{Y^k}(M)$,
\begin{equation}\label{eq:13a}
|\mathcal{T}(y) - \mathcal{T}(v)|
\leq k M^{k-1} \  \| T \| \  \| y-v \|_{Y^k}.
\end{equation}
\end{lemma}
The proof is given in the Appendix.

\begin{lemma} \label{lemma:continuousTensors2}
Let $\mT\in\mM(Y^k,\R).$ Then, it is also infinitely many times differentiable. In particular, for all $y=(y_1,...,y_k)
\in Y^k$ and $z=(z_1,...,z_k) \in Y^k$,
\begin{equation} \label{eq:DifferentiabilityTensors}
D\mathcal{T}(y_1,...,y_k)(z_1,...,z_k)
= \sum_{i=1}^k \mathcal{T}(y_1,...,y_{i-1},z_i,y_{i+1},...,y_k).
\end{equation}
Moreover, for all $M>0$, for all $y$ and $\tilde{y} \in B_{Y^k}(M)$,
\begin{align}
& \big| D\mathcal{T}(y)z \big| \leq k M^{k-1} \| z \|_{Y^k} \label{eq:DifferentiabilityTensors2} \\
& \big| D\mathcal{T}(\tilde{y}) z - D\mathcal{T}(y) z \big|
\leq k (k-1) M^{k-2} \  \| \mathcal{T} \| \  \| \tilde{y}-y \|_{Y^k} \
\| z \|_{Y^k}.\label{eq:DifferentiabilityTensors3}
\end{align}
\end{lemma}

\begin{proof}
The Fr\'echet differentiability of $\mT \in \mM(Y^k,\R),$ as well as formula
\eqref{eq:DifferentiabilityTensors} follow from \eqref{eq:13a},
taking $v_1= y_1+ \theta z_1$,...,$v_k= y_k + \theta z_k$. Formula
\eqref{eq:DifferentiabilityTensors2} follows directly from formula \eqref{eq:DifferentiabilityTensors}.
Formula \eqref{eq:DifferentiabilityTensors3} follows from Lemma
\ref{lemma:continuousTensors1}, from \eqref{eq:DifferentiabilityTensors}, and
from the following relation:
\begin{equation*}
\| \mathcal{T}(\cdot,...,\cdot,z_i,\cdot,...,\cdot) \| = \| z_i \|_Y \, \|
\mathcal{T} \|.
\end{equation*}
Finally, one can prove by induction that $\mathcal{T}$ is infinitely  many times
differentiable, observing that $D\mathcal{T}(y_1,...,y_k)(z_1,...,z_k)$ can be
written as a sum of bounded multilinear forms.
\end{proof}

The following lemma provides a useful chain rule.

\begin{lemma} \label{lemma:ChainRule}
Let $f \in W^{1,1}(0,\infty;Y^k)$ and $\mT\in\mM(Y^k,\R).$ Then,
$F:=\mathcal{T} \circ f$ lies in $W^{1,1}(0,\infty)$ and satisfies
\begin{equation*}
F'(t)= D\mathcal{T}(f(t)) f'(t), \quad \text{for a.e.\@ $t \geq 0$}.
\end{equation*}
\end{lemma}

\begin{proof}
Using the continuous embedding of $W^{1,1}(0,\infty;Y^k)$ in
$L^{\infty}(0,\infty;Y^k)$, we first obtain that
\begin{align*}
& \int_0^\infty |F(t)| \dd t
\leq \| \mathcal{T} \| \; \| f \|_{L^1(0,\infty;Y^k)} \; \|
f\|_{L^\infty(0,\infty;Y^k)}^{k-1} < \infty \\
& \int_0^\infty D\mathcal{T} \circ f(t) f'(t) \dd t \leq
k \| \mathcal{T} \| \; \| f \|_{L^\infty(0,\infty;Y^k)}^{k-1} \; \| f'
\|_{L^1(0,\infty;Y^k)}
< \infty.
\end{align*}
Therefore, $F \in L^1(0,\infty)$ and $D\mathcal{T}\circ f(\cdot) f'(\cdot) \in
L^1(0,\infty)$. It remains to prove that $D \mathcal{T} \circ f(\cdot)
f'(\cdot)$ is the derivative of $F$ in the sense of distributions.

Let $(f_n)_{n \in \mathbb{N}}$ be a sequence in $C^1(0,\infty;Y^k)$, with limit
$f$ in $W^{1,1}(0,\infty;Y^k)$.  Let $\phi \in C_c^\infty(0,\infty)$ be a test
function. By the chain rule, we have
\begin{equation} \label{eq:ChainRule1}
\int_0^\infty \mathcal{T} \circ f_n (t) \phi'(t) \dd t
= -\int_0^\infty D \mathcal{T} \circ f_n (t) f_n'(t) \phi(t) \dd t.
\end{equation}
Using the continuous embedding  of $W^{1,1}(0,\infty;Y^k)$ in
$L^{\infty}(0,\infty;Y^k)$, we obtain that $(f_n)_{n \in \mathbb{N}}$ is
bounded
in $L^\infty(0,\infty;Y^k)$. Let $M>0$ be an upper bound of $\| f_n
\|_{L^\infty(0,\infty; Y^k)}$ and $\| f \|_{L^\infty(0,\infty;Y^k)}$.
By Lemma \ref{lemma:continuousTensors1},
\begin{equation*}
\Big| \int_0^\infty (\mathcal{T} \circ f_n - \mathcal{T} \circ f ) \phi'(t) \dd
t \, \Big|
\leq k M^{k-1} \; \| \mathcal{T} \| \; \| f_n - f \|_{L^1(0,\infty;Y^k)} \|
\phi' \|_{L^\infty(0,\infty)} \underset{n \to \infty}{\longrightarrow} 0.
\end{equation*}
By Lemma \ref{lemma:continuousTensors2},
\begin{align*}
& \Big| \int_0^\infty \Big( D \mathcal{T} \circ f_n(t) f_n'(t) - D\mathcal{T}
\circ f(t) f'(t) \Big) \phi(t) \dd t \, \Big| \\
\leq \ & \int_0^\infty \big| \big( D\mathcal{T} \circ f_n(t) - D \mathcal{T}
\circ f(t) \big) f_n'(t) \big| \; |\phi(t)| \dd t  + \int_0^\infty \big|
D\mathcal{T} \circ f(t) \big( f_n'(t)- f'(t) \big) \big| \; | \phi(t) | \dd t
\\
\leq \ & k(k-1)M^{k-2} \| \mathcal{T} \| \; \| f_n - f \|_{L^\infty(0,\infty;
Y^k)}^{k-1} \| f_n' \|_{L^1(0,\infty;Y^k)} \; \| \phi \|_{L^\infty(0,\infty)}
\\
& \qquad + k M^{k-1} \| \mathcal{T} \| \; \| f_n'-f' \|_{L^1(0,\infty; Y^k)}
\; \| \phi \|_{L^\infty(0,\infty)} \underset{n \to \infty}{\longrightarrow}
0.
\end{align*}
Passing to the limit in \eqref{eq:ChainRule1}, we obtain that
\begin{equation*}
\int_0^\infty \mathcal{T} \circ f (t) \phi'(t) \dd t
= -\int_0^\infty D \mathcal{T} \circ f (t) f'(t) \phi(t) \dd t,
\end{equation*}
which justifies that  $F$ is differentiable in the sense of distributions, with
$F'(\cdot)= D\mathcal{T} \circ f(\cdot) f'(\cdot)$. This concludes the proof.
\end{proof}

\section{Derivation of a generalized Lyapunov equation}
\label{sectionDerivationLyapunov}

The goal of this section is to prove that the derivatives of $\mathcal{V}$ at 0 of order three and more, provided that they exist, are solution to a linear equation, that we call
generalized Lyapunov equation.
The existence of a unique solution to this equation and its use for approximating $\mathcal{V}$
and designing feedback laws
will be discussed in the following sections.
Rather than postulating this equation, we derive it from the HJB equation
under the assumption that
$\mathcal{V}$ is $(k+1)$-times Fr\'echet differentiable in $Y$, with $k \geq 3$,
and under a continuity assumption for optimal controls.
We stress that the assumptions on $\mathcal{V}$, in particular the differentiability at 0, are only used to obtain the generalized Lyapunov equation. The results obtained in the following sections do not rely on this assumption.

\subsection{Derivation of the HJB equation}

We prove in this subsection that the value function $\mathcal{V}$ is a solution the Hamilton-Jacobi-Bellman equation (HJB), under the assumption that $\mathcal{V}$ is continuously differentiable and under a continuity assumption for optimal controls.

Following standard arguments, it can be verified that the dynamic programming
principle for the infinite horizon problem holds: for all $y_0 \in Y$, for all $\tau >0$,
\begin{equation}\label{eq2.3}
\mathcal{V}(y_0)=\underset{u \in
L^2(0,\tau)}{\inf}\int_{0}^{\tau} \ell (S(u,y_0;t),u(t)) \dd t +
\mathcal{V}(S(u,y_0;\tau)),
\end{equation}
where $\ell(y,u)=\frac{1}{2}\| y \|_Y^2+\alpha u^2$.
Moreover, for $\tau >0$, any control $u \in L^2(0,\infty)$ is a solution to problem \eqref{eqProblem} with initial condition $y_0$ if and only if $u_{|(0,\tau)}$ minimizes the r.h.s.\@ of \eqref{eq2.3} and $u_{|(\tau,\infty)}$ is a solution to problem \eqref{eqProblem} with initial condition $S(u,y_0;\tau)$.

\begin{proposition} \label{prop:HJB}
In addition to \eqref{ass:A1}-\eqref{ass:A3}, assume that there exists an open neighborhood $Y_0$ of the origin in $Y$ which is such that the two following statements hold:
\begin{enumerate}
\item For all $y_0 \in Y_0$, problem \eqref{eqProblem} possesses a solution $u$ which is right-continuous at time 0.
\item The value function is continuously differentiable on $Y_0$.
\end{enumerate}
Then, for all $y \in \mathcal{D}(A) \cap Y_0$, the following Hamilton-Jacobi-Bellman equation holds:
\begin{equation}\label{eq:HJB}
D\mV(y) Ay + \frac{1}{2} \|y \|_Y^2 - \frac{1}{2\alpha} \big( D\mV(y) (Ny+B) \big)^2=0.
\end{equation}
\end{proposition}

\begin{proof} The proof uses standard arguments.
Let $y_0 \in \mD(A) \cap Y_0$ be arbitrary. By assumption, there exists an optimal solution $\bar{u}$ to \eqref{eqProblem} with initial condition $y_0$ which is right-continuous at time 0. Let $u_0$ denote the limit of $\bar{u}$ at time 0. Let $\bar{y}= S(\bar{u},y_0)$ be the associated state.
Our  proof is based on the following relations:
\begin{align}
& D\mathcal{V}(y_0) \big( Ay_0 + (Ny_0 + B) u_0 \big) + \ell(y_0, u_0)= 0, \label{eqRelHJB1} \\
& u_0 \in \text{arg min}_{u \in \R} D\mathcal{V}(y_0) \big( Ay_0 + (Ny_0 + B) u \big) + \ell(y_0, u). \label{eqRelHJB2}
\end{align}

\emph{Step 1:} proof of \eqref{eqRelHJB1}.
By the dynamic programming principle, for all $\tau>0$,
\begin{equation*}
\mathcal{V}(y_0) = \int_{0}^{\tau} \ell(\bar{y}(s),\bar{u}(s)) \dd s +
\mathcal{V}(\bar{y}(\tau)).
\end{equation*}
Thus,
\begin{equation}\label{eq2.5}
\frac{1}{\tau} \int_{0}^{\tau} \ell(\bar{y}(s),\bar{u}(s)) \dd s +
\frac{1}{\tau} \big( \mathcal{V}(\bar{y}(\tau)-\mathcal{V}(y_0) \big)=0.
\end{equation}
For any $T>0$, we have $\bar{y} \in C([0,T];Y)$ and therefore, we can fix $\tau_0 >0$ such that $\bar{y}(\tau) \in Y_0$, for all $\tau \in [0,\tau_0]$.
Relation \eqref{eqRelHJB1} follows then by passing to the limit in \eqref{eq2.5}, in $Y$, when $\tau \to 0$.
By continuity of $\bar{y}$ and $\bar{u}$ at time 0, the first term of the left-hand side of \eqref{eq2.5} clearly converges to $\ell(y_0,u_0)$. To prove the convergence of the second term, we need to prove the differentiability of $\bar{y}$ at time 0 and to establish a chain rule property.
For all $\tau \in (0,\tau_0)$, we have
\begin{equation}\label{eq2.6a}
\frac{1}{\tau} \big( \bar{y}(\tau)-y_0 \big)
= \frac{1}{\tau} \big( e^{A\tau} y_0 - y_0 \big)
+ \frac{1}{\tau} \int_{0}^{\tau} e^{A(\tau-s)} \underbrace{\big( (N\bar{y}(s)+B)\bar{u}(s) \big) }_{=: f(s)} \dd s.
\end{equation}
The first term of the r.h.s.\@ converges to $Ay_0$. Regarding the second one, observe 
first that by Lemma \ref{le2.2}, $\bar{y} \in C([0,\tau_0];V)$, therefore, since 
$\bar{u}$ is right-continuous and $N \in \mathcal{L}(V,Y)$, the function $f\colon s \geq 0 
\mapsto f(s) \in Y$ is right-continuous at time 0. We have
\begin{equation} \label{eq2.6b}
\frac{1}{\tau} \int_0^\tau e^{A(\tau-s)} f(s) \dd s - f(0)
= \underbrace{\frac{1}{\tau} \int_0^\tau e^{A(\tau-s)} \big( f(s)-f(0) \big) \dd s}_{=:(a)}
+ \underbrace{\frac{1}{\tau} \int_0^\tau \big( e^{A(\tau-s)} f(0) \big) - f(0) \dd s}_{=:(b)}.
\end{equation}
Since $A$ generates an analytic semigroup (see Remark \ref{KK1}), there exist 
$M>0$ and $\omega>0$ such that $\| e^{As} \|_{\mathcal{L}(Y)} \leq Me^{\omega 
s}$. We obtain
\begin{equation} \label{eq2.6c}
\| (a) \|_Y \ \leq \ \frac{1}{\tau} \int_0^\tau  \| e^{A(\tau-s)} 
\|_{\mathcal{L}(Y)}  \| f(s)-f(0) \|_Y \dd s 
\ \leq \  M e^{\omega\tau_0} \Big( \sup_{s \in [0,\tau]} \| f(s)-f(0) \|_Y 
\Big)
\underset{\tau \to 0}{\longrightarrow} 0.
\end{equation}
Moreover, for $\tilde{f}(t):= e^{At}f(0)$, it holds that $\tilde{f}\in 
C([0,\tau],Y)$, therefore
\begin{equation} \label{eq2.6d}
\| (b) \|_Y \leq \frac{1}{\tau} \int_0^\tau \| \tilde{f}(\tau-s)- \tilde{f}(0) 
\|_Y \dd s
\leq \max_{s \in [0,\tau]} \| \tilde{f}(s)-\tilde{f}(0) \|_Y
\underset{\tau \to 0}{\longrightarrow} 0.
\end{equation}
Combining \eqref{eq2.6a}-\eqref{eq2.6d}, we obtain that
\begin{equation*}
\frac{1}{\tau} \big( \bar{y}(\tau)- y(0) \big) \underset{\tau \to 0}{\longrightarrow} Ay_0 + (Ny_0 + B) u_0. 
\end{equation*}
We now have
\begin{align*}
& \frac{1}{\tau} \big( \mathcal{V}(\bar{y}(\tau)) - \mathcal{V}(y_0) \big)
- D\mathcal{V}(y_0) (Ay_0 + (Ny_0+B)u_0)
= \\ 
& \qquad \underbrace{\frac{1}{\tau} \int_0^1 \big[ D\mathcal{V}(y_0 + s(\bar{y}(\tau)-y_0)) - D\mathcal{V}(y_0) \big](\bar{y}(\tau)-y_0) \dd s}_{=:(c)} \\
& \qquad \qquad + D\mathcal{V}(y_0) \Big( \frac{\bar{y}(\tau)-y_0}{\tau} - (Ay_0 + (Ny_0+B)u_0) \Big).
\end{align*}
Clearly, the second term of the r.h.s.\@ converges to 0. Using the continuity of $D\mathcal{V}$, and the fact that $(\bar{y}(\tau)-y_0)/\tau$ is bounded, we obtain
\begin{equation*}
\| (c) \|_Y
\leq \Big( \max_{z \in B_Y(\| \bar{y}(\tau)-y_0 \|_Y)} \| D\mathcal{V}(y_0 + z )- D\mathcal{V}(y_0) \|
\Big)
\Big\| \frac{1}{\tau} \big( \bar{y}(\tau)-y_0 \big) \Big\|_Y 
\underset{\tau \to 0}{\longrightarrow} 0.
\end{equation*}
Passing to the limit in \eqref{eq2.5}, we obtain: $\ell(y_0,u_0) + D\mathcal{V}(y_0)(Ay_0 + (Ny_0+B)u_0)= 0$, which proves \eqref{eqRelHJB1}.

\emph{Step 2:} proof of \eqref{eqRelHJB2} and conclusion.
Let $u \in \R$ and let $\tilde{u}$ be the piecewise constant control equal to $u$ on $(0,1)$ and equal to 0 on $(1,\infty)$. Let $\tilde{y}= S(y_0,\tilde{u})$.
Then, by \eqref{eq2.3}, for all $\tau \in (0,1)$,
\begin{equation*}
\frac{1}{\tau} \int_{0}^{\tau} \ell(\tilde{y}(s),u) \dd s +
\frac{1}{\tau} \big( \mathcal{V}(\tilde{y}(\tau)-\mathcal{V}(y_0) \big) \geq 0.
\end{equation*}
We can pass to the limit (when $\tau \to 0$) with exactly the same arguments as the ones used in the first part of the proof. We therefore obtain
\begin{equation*}
D\mathcal{V}(y_0) \big( Ay_0 + (Ny_0 + B) u \big) + \ell(y_0, u) \geq 0.
\end{equation*}
Since the l.h.s.\@ in the above expression is equal to 0 for $u= u_0$, we deduce that it reaches its minimum 0 at $u= u_0$. The l.h.s.\@ being linear-quadratic with respect to $u$, the following relation can easily be obtained:
\begin{equation} \label{eqFeedbackGeneral}
u_0= - \frac{1}{\alpha} D\mathcal{V}(y_0)(Ny_0 + B).
\end{equation}
Equation \eqref{eq:HJB} follows then from \eqref{eqRelHJB1} and \eqref{eqFeedbackGeneral}.
\end{proof}

\subsection{A generalized operator Lyapunov equation}
\label{subsecGenLyapunov}

We prove in Theorem \ref{thm:DifferentiabilityImpliesLyapunov} that if $\mathcal{V}$ is $(k+1)$-times differentiable, then $D^k \mathcal{V}(0)$ is a solution to a generalized Lyapunov equation, by differentiating the HJB equation $k$-times.
Note that in this subsection, the $k$-th derivative $D^k \mathcal{V}(0)$ is
represented by a multilinear form in $\mM(Y^k,\R).$

\paragraph{The case $k=3$.}

We assume that $\mathcal{V}$ is four times Fr\'echet differentiable on $Y$ and that the assumptions of Proposition \ref{prop:HJB} hold. Note that the differentiability on $Y$ implies the differentiability on $\mathcal{D}(A)$. Differentiating the HJB equation \eqref{eq:HJB} a first time with respect to
$y$ in the direction $z_1 \in \mathcal{D}(A)$ yields
\begin{align*}
& D^2 \mathcal{V}(y)(Ay,z_1) + D\mathcal{V}(y)A z_1 + \langle y, z_1 \rangle_Y \\
& \quad - \frac{1}{\alpha} \big( D^2 \mathcal{V}(y)(Ny+B,z_1) +
D\mathcal{V}(y)Nz_1 \big) \big( D\mathcal{V}(y)(Ny+B) \big) = 0.
\end{align*}
Differentiating a second time with  respect to $y$ in the direction $z_2 \in
\mathcal{D}(A)$, we obtain
\begin{align}
& D^3 \mathcal{V}(y)(Ay,z_1,z_2) + D^2 \mathcal{V}(y)(Az_2,z_1) + D^2
\mathcal{V}(y)(Az_1,z_2) + \langle z_1, z_2 \rangle_Y \notag \\
& \quad - \frac{1}{\alpha} \Big( D^2 \mathcal{V}(y)(Ny+B, z_1) +
D\mathcal{V}(y)Nz_1 \Big)
\Big( D^2 \mathcal{V}(y)(Ny + B,z_2) + D\mathcal{V}(y)Nz_2 \Big) \notag \\
& \quad - \frac{1}{\alpha} \Big( D^3 \mathcal{V}(y)(Ny+B,z_1,z_2) + D^2
\mathcal{V}(y)(Nz_2,z_1) + D^2 \mathcal{V}(y)(Nz_1,z_2) \Big)
\Big( D\mathcal{V}(y)(Ny+B) \Big) \notag \\
& \quad = 0. \label{eq:HJBTwiceDiff}
\end{align}
Observing that  $\mathcal{V}(y) \geq 0$ for all $y$ and  that $\mathcal{V}(0)=
0$, we deduce that $D\mathcal{V}(0)= 0$. Taking $y= 0$ in the above equation
and representing $D^2 \mathcal{V}(0)$ as nonnegative self-adjoint operator
$\Pi=\Pi^* \in
\mathcal{L}(Y)$ such that $D^2\mathcal{V}(0)(z_1,z_2)= \langle z_1, \Pi z_2
\rangle_Y$ for all $z_1,z_2 \in \mD(A)$, we obtain
\begin{equation} \label{eq:algebraicRiccati}
\langle A^* \Pi z_1,z_2 \rangle + \langle \Pi A z_1,z_2 \rangle + \langle
z_1,z_2\rangle - \frac{1}{\alpha}  (B^* \Pi z_1)(B^* \Pi z_2)= 0.
\end{equation}
Equation \eqref{eq:algebraicRiccati} is the algebraic operator Riccati
equation, see e.g.\@ \cite{CurZ95,LasT00}. Throughout the rest of the paper, we
assume, on top of assumptions \eqref{ass:A1}-\eqref{ass:A3} that
\begin{equation}\label{ass:A4} \tag{A4}
  \exists F \in \mL(Y,\R) \text{ such that the semigroup $e^{(A+BF)t}$
is exponentially stable on } Y.
\end{equation}
Since the pair $(A,I)$ is exponentially detectable on $Y$, it follows from \cite[Theorem
6.2.7]{CurZ95} that \eqref{eq:algebraicRiccati} has a unique
nonnegative stabilizing solution $\Pi \in \mL(Y)$. Accordingly, we define the
operator $A_{\Pi}$  as
follows:
\begin{equation*}
\begin{aligned}
 A_{\Pi} \colon & \mD(A_\Pi) \subset Y \to Y, \quad \mD(A_\Pi) = \big\{y \in
L^2(\Omega)\ | \ Ay - \frac{1}{\alpha}BB^*\Pi\in Y
\big\}, \\
 & y \mapsto A_{\Pi}y:= Ay - \frac{1}{\alpha} BB^*\Pi y .
\end{aligned}
\end{equation*}
In particular, since $\Pi$ is stabilizing, we know that the semigroup
$e^{A_\Pi t}$ is exponentially stable on $Y.$ Moreover, since $BB^*\Pi \in
\mL(Y),$ by a perturbation result for analytic semigroups \cite{Paz83}, as in
Remark \ref{KK1} we can choose $\tilde{\lambda}\ge 0$ such that $-\tilde{A}_0
=-A_\Pi+\tilde{\lambda}I$ is maximal accretive. Endowing $\mD(-\tilde{A}_0)$ and
$\mD(-A_0)$ with their graph norms, we have that the identity operator between  these spaces is a homeomorphism
 $\mD(-\tilde{A}_0)\cong \mD(-A_0).$ Consequently, the
interpolation spaces defined by the method of traces \cite[Part II, Chapter
1, Section 2]{Benetal07} are homeomorphic and we thus obtain
\begin{align*}
  [\mD(-\tilde{A}_0),Y)]_{\frac{1}{2}} = [\mD(-\tilde{A}_0^*),Y]_{\frac{1}{2}}
= V.
\end{align*}
We continue by differentiating a third time with respect to $y$ in the
direction $z_3 \in \mathcal{D}(A),$  which for $y=0$ leads us to:
\begin{align*}
& D^3 \mathcal{V}(0)(Az_3,z_1,z_2) + D^3\mathcal{V}(0)(Az_2,z_1,z_3)+ D^3
\mathcal{V}(0)(Az_1,z_2,z_3) \\
& \quad - \frac{1}{\alpha} \Big( D^3V(0)(B,z_1,z_3) + D^2
\mathcal{V}(0)(Nz_3,z_1) +  D^2 \mathcal{V}(0)(Nz_1,z_3) \Big)
\Big( D^2 \mathcal{V}(0)(B,z_2) \Big) \\
& \quad - \frac{1}{\alpha} \Big( D^3\mathcal{V}(0)(B,z_2,z_3) + D^2
\mathcal{V}(0)(Nz_3,z_2) + D^2 \mathcal{V}(0)(Nz_2,z_3) \Big)
\Big( D^2 \mathcal{V}(0)(B,z_1) \Big) \\
& \quad - \frac{1}{\alpha} \Big( D^3 \mathcal{V}(0)(B,z_1,z_2) + D^2
\mathcal{V}(0)(Nz_2,z_1) + D^2 \mathcal{V}(0)(Nz_1,z_2) \Big)
\Big( D^2\mathcal{V}(0)(B,z_3) \Big) \\
& \quad = 0.
\end{align*}
We can already observe that this equation is a linear equation with respect to
$D^3 \mathcal{V}(0)$. Moreover, using the symmetry of the derivatives, we can
re-write it in the following form:
\begin{equation} \label{eq:lyapOrder3}
D^3 \mathcal{V}(0)(A_\Pi z_1,z_2,z_3) + D^3\mathcal{V}(0)(z_1,A_\Pi z_2,z_3) +
D^3 \mathcal{V}(0)(z_1,z_2,A_\Pi z_3) = \frac{1}{2\alpha}
\mathcal{R}_3(z_1,z_2,z_3),
\end{equation}
where the multilinear form $\mathcal{R}_3\colon Y^3 \rightarrow \R$ is defined
by
\begin{align*}
\mathcal{R}_3(z_1,z_2,z_3)= \ &
2 (\Pi B, z_1) \big[ (\Pi z_2, Nz_3) + (\Pi z_3, N z_2) \big]
+ 2 (\Pi B, z_2) \big[ (\Pi z_1, Nz_3) + (\Pi z_3, N z_1) \big] \\
& \qquad + 2 (\Pi B, z_3) \big[ (\Pi z_1, Nz_2) + (\Pi z_2, N z_1) \big].
\end{align*}

\paragraph{Lyapunov equation: general case.}

The derivation of the Lyapunov equation, for a general $k \geq 3$, requires some
symmetrization techniques for multilinear forms. For $i$ and $j \in \mathbb{N}$, we make
use of the following set of permutations:
\begin{equation*}
S_{i,j}= \big\{ \sigma \in S_{i+j} \,|\, \sigma(1)< ... < \sigma(i) \text{ and }
\sigma(i+1) < ... < \sigma(i+j) \big \},
\end{equation*}
where $S_{i+j}$ is the set of permutations of $\{ 1,...,i+j \}$.
A permutation $\sigma \in S_{i,j}$ is uniquely defined by the subset
$\{\sigma(1),...,\sigma(i) \}$, therefore, the cardinality of $S_{i,j}$ is equal
to the number of subsets of cardinality $i$ of $\{ 1,...,i+j\}$, that is to say
\begin{equation*}
|S_{i,j}|= \begin{pmatrix} i+j \\ i \end{pmatrix}.
\end{equation*}
Let us give an example. Representing a permutation $\sigma \in S_4$ by the
vector $(\sigma(1),...,\sigma(4))$, we have:
\begin{align*}
S_{2,2}= \ & \big\{ \sigma \in S_4 \,|\, \sigma(1) < \sigma(2) \text{ and }
\sigma(3) < \sigma(4) \big\} \\
= \ & \big\{ (1,2,3,4),  (1,3,2,4), (1,4,2,3), (2,3,1,4), (2,4,1,3), (3,4,1,2)
\big\}.
\end{align*}
Let $\mathcal{T}$ be a multilinear form of order $i+j$. We denote by
$\text{Sym}_{i,j}(\mathcal{T})$ the multilinear form defined by
\begin{equation} \label{eq:defSym}
\text{Sym}_{i,j}(\mathcal{T})(z_1,...,z_{i+j}) = \begin{pmatrix} i+j \\ i
\end{pmatrix}^{-1} \Big[ \sum_{\sigma \in S_{i,j}} \mathcal{T}
(z_{\sigma(1)},...,z_{\sigma(i+j)} ) \Big].
\end{equation}

The two following lemmas contain the main properties related to this specific symmetrization technique which will be needed. Their proofs are given in the Appendix. Lemma \ref{lemma:leibnitz} is a general Leibnitz formula for the differentiation of the product of two functions. Lemma \ref{lemma:symmetry} is a symmetry property.

\begin{lemma} \label{lemma:leibnitz}
Let $f\colon Y \rightarrow \R$ and $g \colon Y \rightarrow \R$ be two $k$-times continuously differentiable functions. Then, for all $k \geq 1$, for all $y \in Y$,
\begin{equation} \label{eq:leibnitz0}
D^k \big[f(y) g(y) \big]
= \sum_{i=0}^k \begin{pmatrix} k \\ i \end{pmatrix} \text{Sym}_{i,k-i} \big( D^i f(y) \otimes D^{k-i} g(y) \big).
\end{equation}
\end{lemma}

\begin{lemma} \label{lemma:symmetry}
Let $\mathcal{T}_1 \in \mathcal{M}(Y^i,\R)$ and $\mathcal{T}_2 \in \mathcal{M}(Y^j,\R)$. Then, for all $y \in Y$,
\begin{equation*}
\text{Sym}_{i,j}(\mathcal{T}_1 \otimes \mathcal{T}_2 )(y^{\otimes (i+j)})=
\mathcal{T}_1(y^{\otimes i}) \mathcal{T}_2(y^{\otimes j}).
\end{equation*}
Moreover, if $\mathcal{T}_1$ and $\mathcal{T}_2$ are symmetric, then $\text{Sym}_{i,j}(\mathcal{T}_1 \otimes \mathcal{T}_2)$ is also symmetric.
\end{lemma}

We are now ready to derive the generalized Lyapunov equation.

\begin{theorem} \label{thm:DifferentiabilityImpliesLyapunov}
Let $k \geq 3$. Assume that  $\mathcal{V}\colon Y \rightarrow \R$ is
$(k+1)$-times Fr\'echet differentiable in a neighborhood of 0 and that the assumptions of Proposition \ref{prop:HJB} hold.
Then for all $z_1$,...,$z_k \in \mathcal{D}(A)$,
\begin{equation} \label{eq:Lyapunov1}
\sum_{i=1}^k
\mathcal{D}^k\mathcal{V}(0)(z_1,...,z_{i-1},A_{\Pi}z_i,z_{i+1},...,z_k)
= \frac{1}{2\alpha} \mathcal{R}_k(z_1,...,z_k),
\end{equation}
where the multilinear form $\mathcal{R}_k\colon Y^k \rightarrow \R$ is given by:
\begin{align*}
& \mathcal{R}_k= 2k(k-1) \text{\emph{Sym}}_{1,k-1}\big( \mathcal{C}_1 \otimes
\mathcal{G}_{k-1} \big)  + \sum_{i=2}^{k-2} \begin{pmatrix} k \\ i \end{pmatrix}
\text{\emph{Sym}}_{i,k-i} \big( (\mathcal{C}_i + i \mathcal{G}_i) \otimes
(\mathcal{C}_{k-i} + (k-i) \mathcal{G}_{k-i}) \big), \\
& \text{where: } \quad \begin{cases} \begin{array}{rl}
 \mathcal{C}_i(z_1,...,z_i) = & \displaystyle{D^{i+1}
\mathcal{V}(0)(B,z_1,...,z_i)}, \ \text{for $i=1,...,k-2$} \\
\displaystyle \mathcal{G}_i(z_1,...,z_i) = & \displaystyle{\frac{1}{i} \Big[
\sum_{j=1}^i D^i \mathcal{V}(0)(z_1,...,z_{j-1},Nz_j,z_{j+1},...,z_i) \Big],} \ \text{for $i=1,...,k-1$.}
\end{array} \end{cases}
\end{align*}
\end{theorem}

\begin{remark}
The meaning of the expression on the left-hand side of \eqref{eq:Lyapunov1} is the following:
\begin{align*}
& \sum_{i=1}^k D^k \mathcal{V}(0)(z_1,...,z_{i-1},A_{\Pi} z_i, z_{i+1},...,z_k)
= D^k \mathcal{V}(0)(A_{\Pi}z_1,z_2,...,z_k) \\
& \qquad + D^k \mathcal{V}(0)(z_1,A_{\Pi}z_2,z_3,...,z_k)
+ ...
+ D^k \mathcal{V}(0)(z_1,...,z_{k-1},A_{\Pi}z_k).
\end{align*}
\end{remark}

\begin{proof}[Proof of Theorem \ref{thm:DifferentiabilityImpliesLyapunov}]
We differentiate  the HJB equation  $k$ times.
First observe that since $k \geq 3$,
\begin{equation} \label{eqDerivative0}
D^k (\|y\|_Y^2)= 0.
\end{equation}
We then have
\begin{align}
D^k\big[ D\mathcal{V}(y)(Ay) \big](z_1,...,z_k)
= & D^{k+1} \mathcal{V}(y)(Ay,z_1,...,z_k) \notag \\
& \qquad + \sum_{i=1}^k D^k \mathcal{V}(y)
(z_1,...,z_{i-1},Az_i,z_{i+1},...,z_k). \label{eqDerivative1}
\end{align}
Therefore, the $k$-th derivative of $y \mapsto D\mathcal{V}(y)(Ay)$, evaluated
at $y= 0$, is given by
\begin{equation} \label{eqDerivative2}
\sum_{i=1}^k D^k \mathcal{V}(0) (z_1,...,z_{i-1},Az_i,z_{i+1},...,z_k).
\end{equation}
For all $y \in \mathcal{D}(A)$ we set $\mathcal{W}(y)= D\mathcal{V}(y)(Ny+B)$.
It remains to compute the $k$-th derivative of $y \in \mathcal{D}(A) \mapsto
\mathcal{W}(y)^2$ at $y= 0$.
Similarly to \eqref{eqDerivative1},
\begin{align*}
D^i \mathcal{W}(y)(z_1,...,z_i)
= & D^{i+1} \mathcal{V}(y)(Ny+B,z_1,...,z_i) \\
& \qquad + \sum_{j=1}^i D^i \mathcal{V}(y)
(z_1,...,z_{j-1},Nz_j,z_{j+1},...,z_i),
\end{align*}
and therefore,
\begin{equation} \label{eqDerivative3}
D^i \mathcal{W}(0)= \mathcal{C}_i+ i \mathcal{G}_i.
\end{equation}
Using Lemma \ref{lemma:leibnitz} and observing that $D^0 \mathcal{W}(0)=
\mathcal{W}(0)= 0$, we obtain
\begin{align} 
D^k \big[ \mathcal{W}(y)^2 \big]_{| y= 0}=\ & \sum_{i=0}^k \begin{pmatrix} k \\ i \end{pmatrix} \text{Sym}_{i,k-i} \big(
D^i \mathcal{W}(0) \otimes D^{k-i} \mathcal{W}(0) \big) \notag \\
=\ & \sum_{i=1}^{k-1} \begin{pmatrix} k \\ i \end{pmatrix}
\text{Sym}_{i,k-i} \big( (\mathcal{C}_i + i \mathcal{G}_i) \otimes (\mathcal{C}_
{k-i} + (k-i) \mathcal{G}_{k-i} ) \big). \label{eqDerivative4}
\end{align}
We compute now the summands of the above expression for $i=1$ and $i=k-1$.
Note first that
\begin{equation*}
\mathcal{G}_1= 0 \quad \text{and} \quad
\mathcal{C}_1(z)= B^* \Pi z.
\end{equation*}
Therefore,
\begin{align}
& \text{Sym}_{1,k-1} \big( (\mathcal{C}_1 + \mathcal{G}_1) \otimes (\mathcal{C}_
{k-1}+ (k-1) \mathcal{G}_{k-1}) \big) \notag \\
& \qquad = \text{Sym}_{1,k-1} \big( \mathcal{C}_1 \otimes \mathcal{C}_{k-1}
\big)
+ (k-1) \text{Sym}_{1,k-1} \big( \mathcal{C}_1 \otimes \mathcal{G}_{k-1} \big)
\label{eqDerivative5}
\end{align}
and moreover
\begin{align}
\text{Sym}_{1,k-1} \big( \mathcal{C}_1 \otimes \mathcal{C}_{k-1}
\big)(z_1,...,z_k)
= \ & \sum_{j=1}^k \mathcal{C}_1(z_j)
\mathcal{C}_{k-1}(z_1,...,z_{j-1},z_{j+1},...,z_k) \notag \\
= \ & \sum_{j=1}^k B^* \Pi z_j D^k \mathcal{V}(0)(z_1,...,z_{j-1},B,z_{j+1},...,z_k) \notag \\
= \ & \sum_{j=1}^k D^k \mathcal{V}(0)(z_1,...,z_{j-1},BB^* \Pi z_j,
z_{j+1},...,z_k). \label{eqDerivative6}
\end{align}
Combining \eqref{eqDerivative4}, \eqref{eqDerivative5}, and
\eqref{eqDerivative6}, we obtain
\begin{align}
 D^k \big[ \mathcal{W}(y)^2 \big]_{| y= 0}(z_1,...,z_k) = \ & \sum_{j=1}^k D^k
\mathcal{V}(0)(z_1,...,z_{j-1},BB^* \Pi z_j, z_{j+1},...,z_k) \notag \\
 & \quad + \sum_{i=2}^{k-2} \begin{pmatrix}k \\ i \end{pmatrix}
\text{Sym}_{i,k-i} \big( (\mathcal{C}_i + i \mathcal{G}_i) \otimes (\mathcal{C}_
{k-i} + (k-i) \mathcal{G}_{k-i} ) \big)(z_1,...,z_k) \notag \\
& \quad + 2k(k-1) \text{Sym}_{1,k-1} \big( \mathcal{C}_1 \otimes
\mathcal{G}_{k-1} \big)(z_1,...,z_k).
\label{eqDerivative7}
\end{align}
From \eqref{eqDerivative0}, \eqref{eqDerivative2}, and
\eqref{eqDerivative7}, we deduce \eqref{eq:Lyapunov1}.
\end{proof}

\section{Construction of the polynomial approximation}
\label{sectionWellPosednessLyapunov}

In this section, we construct a sequence $(\mT_k)_{k
\geq 2}$,  with $\mT_k \in \mM(Y^k,\R)$, which enables us to obtain a polynomial
approximation of the value function $\mathcal{V}$.
For all $k \geq 3$,  $\mathcal{T}_k$ is the unique solution to a multilinear
equation, with a right-hand side which depends explicitly on $N$, $B$, and
$\mathcal{T}_2$,...,$\mathcal{T}_{k-1}$. This multilinear equation is suggested
by the structure of \eqref{eq:Lyapunov1}. The existence will be obtained under the generic
assumptions \eqref{ass:A1}-\eqref{ass:A4}.

We start with an existence result for multilinear equations with particular
right-hand sides, which will be relevant once we turn to  \eqref{eq:Lyapunov1}.

\begin{theorem} \label{thm:existenceUniquenessLyapunov}
Let $k \geq 2$. For $1\le i < j \le k,$ let $\mathcal{R}_{ij} \in \mM(Y^k,\R).$
Then, there exists a unique $\mT \in \mM(Y^k,\R)$ such that for all
$(z_1,...,z_k) \in \mathcal{D}(A)^k$,
\begin{equation} \label{eq:Lyapunov2}
\sum_{i=1}^k \mathcal{T}(z_1,...,z_{i-1},A_\Pi z_i,z_{i+1},...,z_k)
= \mathcal{R}(z_1,...,z_k),
\end{equation}
where:
\begin{equation*}
\mathcal{R}(z_1,...,z_k)
= \sum_{1 \leq i < j \leq k}
\mathcal{R}_{ij}(z_1,...,z_{i-1},Nz_i,z_{i+1},..., z_{j-1}, Nz_j,
z_{j+1},...,z_k).
\end{equation*}
Moreover, if $\mathcal{R}$ is symmetric, then $\mT$ is also symmetric.
\end{theorem}

\begin{proof}
\emph{Part 1: existence}. For all $(z_1,...,z_k) \in Y^k$, we define:
\begin{equation*}
\mathcal{T}(z_1,...,z_k)
= - \int_0^\infty \mathcal{R}(e^{A_\Pi t} z_1,..., e^{A_\Pi t} z_k) \dd t.
\end{equation*}
Let us justify the well-posedness of $\mathcal{T}$. All along the article, the constant $C$ is a generic constant whose value can change. We have
\begin{align*}
& \int_0^\infty \big| \mathcal{R}_{12}(N e^{A_\Pi t} z_1, N e^{A_\Pi t} z_2,
e^{A_\Pi t} z_3,..., e^{A_\Pi t} z_k) \big| \dd t  \\
& \quad \leq C \int_0^\infty \Big[ \| Ne^{A_\Pi t} z_1 \|_Y \  \| Ne^{A_\Pi t}
z_2 \|_Y \
\prod_{i=3}^k \| e^{A_\Pi t} z_i \|_Y \Big] \dd t  \\
& \quad \leq C \int_0^\infty \Big[ \|  e^{A_\Pi t} z_1 \|_V \  \|  e^{A_\Pi t}
z_2 \|_V \
\prod_{i=3}^k \| e^{A_\Pi t} z_i \|_Y \Big] \dd t .
\end{align*}
Here, the last step follows from the fact that $N \in \mL(V,Y)$. Using the generalized
H\"older inequality, we obtain
\begin{align*}
& \int_0^\infty \big| \mathcal{R}_{12}(N e^{A_\Pi t} z_1, N e^{A_\Pi t} z_2,
e^{A_\Pi t} z_3,..., e^{A_\Pi t} z_k) \big| \dd t \\
& \quad \leq  C \| e^{A_\Pi \cdot } z_1 \|_{L^2(0,\infty;V)}
\| e^{A_\Pi \cdot } z_1 \|_{L^2(0,\infty;V)}
\prod_{i=3}^k \| e^{A_\Pi \cdot} z_i \|_{L^\infty(0,\infty;Y)} .
\end{align*}
Since the semigroup $e^{A_\Pi t}$ is analytic and exponentially stable on $Y$, it follows from \cite[Theorem 2.2, Part II, Chapter 3]{Benetal07} that
\begin{align}\label{eq:Lyapounov2c}
 \int_0^\infty \big| \mathcal{R}_{12}(N e^{A_\Pi t} z_1, N e^{A_\Pi t} z_2,
e^{A_\Pi t} z_3,..., e^{A_\Pi t} z_k) \big| \dd t \leq \ &  C \prod_{i=1}^k \|
z_i \|_Y.
\end{align}
The same estimate can be derived for the other terms of $\mathcal{R}$. It
follows that
\begin{equation} \label{eq:Lyapounov2a}
\int_0^\infty | \mathcal{R}(e^{A_\Pi t} z_1,...,e^{A_\Pi t} z_k) | \dd t
\leq C \prod_{i=1}^k \| z_i \|_Y,
\end{equation}
which proves that $\mathcal{T}$ is well-defined on $Y^k$. If $\mathcal{R}$ is symmetric, then $\mathcal{T}$ is also symmetric, by \eqref{eq:Lyapounov2a}.

We next prove that $\mathcal{T}$ is a solution to \eqref{eq:Lyapunov2}. Let us
first assume that $(z_1,...,z_k) \in \mathcal{D}(A^2)^k$ and define
\begin{equation*}
F\colon t \in [0,\infty) \mapsto \mathcal{R}(e^{A_\Pi t} z_1,...,e^{A_\Pi t}
z_k).
\end{equation*}
We already know that $F \in L^1(0,\infty)$, by \eqref{eq:Lyapounov2a}. In fact,
$F \in W^{1,1}(0,\infty)$, with
\begin{equation} \label{eq:Lyapounov2b}
F'(t)= \sum_{i=1}^k \mathcal{R}(e^{A_\Pi t} z_1,..., e^{A_\Pi t} z_{i-1},
e^{A_\Pi t} A_{\Pi} z_i, e^{A_\Pi t} z_{i+1},..., z_k).
\end{equation}
This is seen as follows. For all $i<j$, for $t \in [0,\infty)$, we define
\begin{equation*}
F_{ij}(t)= \mathcal{R}_{ij}(e^{A_\Pi t}z_1,...,e^{A_\Pi t} z_{i-1}, N e^{A_\Pi
t} z_i, e^{A_\Pi t} z_{i+1},...,e^{A_\Pi t} z_{j-1}, N e^{A_\Pi t} z_j, e^{A_\Pi
t} z_{j+1},..., e^{A_\Pi t} z_k),
\end{equation*}
so that $F= \sum_{1 \leq i < j \leq k} F_{ij}$.
To simplify, we focus on $F_{12}$.
By \cite[Theorem 5.1.5]{CurZ95}, $A_{\Pi}^{-1}$ exists and $A_{\Pi}^{-1} \in \mathcal{L}(Y,\mathcal{D}(A))$.
Using the commutativity of
$A_\Pi$, $A_{\Pi}^{-1}$, and $e^{A_\Pi t}$, we find that
\begin{align*}
F_{12}(t)= \ & \mathcal{R}_{12}(NA_\Pi^{-1} e^{A_\Pi t} A_\Pi z_1,
NA_\Pi^{-1} e^{A_\Pi t} A_\Pi z_2,
e^{A_\Pi t}z_3,...,e^{A_\Pi t}z_k)  =  \widehat{\mathcal{R}}_{12} \circ
g_{12}(t),
\end{align*}
where
\begin{align*}
 \widehat{\mathcal{R}}_{12}(y_1,...,y_k)&:=
\mathcal{R}_{12}(NA_{\Pi}^{-1} y_1, NA_\Pi^{-1} y_2, y_3,...,y_k) , \\
g_{12}(t)&:= (e^{A_\Pi t} A_{\Pi} z_1,e^{A_\Pi
t} A_{\Pi} z_2, e^{A_\Pi t} z_3...,e^{A_\Pi t} z_k).
\end{align*}
Since $NA_{\Pi}^{-1} \in \mathcal{L}(Y),$ it follows that
$\widehat{\mathcal{R}}_{12} \in \mM(Y^k,\R).$ Moreover, for $z_i \in \mD(A^2)$
it holds that $A_\Pi e^{A_\Pi \cdot } A_\Pi z_i \in L^1(0,\infty;Y)$ and
hence, $g_{12}\in W^{1,1}(0,\infty;Y^k)$. By Lemma \ref{lemma:ChainRule} we
obtain that $F_{12} \in W^{1,1}(0,\infty)$ and that
\begin{align*}
F'_{12}(t)= \ & \mR_{12} (N e^{A_\Pi t} A_{\Pi} z_1, Ne^{A_\Pi
t} z_2, e^{A_\Pi t} z_3,..., e^{A_\Pi t} z_k) \\
& \quad + \mR_{12} (N e^{A_\Pi t} z_1, Ne^{A_\Pi t} A_{\Pi}
z_2, e^{A_\Pi t} z_3,..., e^{A_\Pi t} z_k) \\
& \quad + \sum_{i=3}^k \mR_{12} (N e^{A_\Pi t} z_1, N
e^{A_\Pi t} z_2, e^{A_\Pi t} z_3,..., e^{A_\Pi t} z_{i-1}, e^{A_\Pi t } A_{\Pi}
z_i, e^{A_\Pi t} z_{i+1},..., e^{A_\Pi t} z_k ).
\end{align*}
Similar formulas can be obtained in the same manner for $F_{ij}$. It follows
that $F \in W^{1,1}(0,\infty)$ and that \eqref{eq:Lyapounov2b} holds. Since
$\mathcal{R}$ is continuous and $ \| e^{A_\Pi t} z_i \|_Y \underset{t \to
\infty}{\longrightarrow} 0$, we deduce $F(t) \underset{t \to
\infty}{\longrightarrow} 0$.
Moreover, $F \in W^{1,1}(0,\infty)$ implies that it is absolutely continuous
and therefore, for all $T \geq 0$,
\begin{equation*}
F(T)-F(0)= \int_0^T F'(t) \dd t.
\end{equation*}
Passing to the limit when $T \to \infty$, we obtain
\begin{align*}
F(0)= \ & -\int_0^{\infty} F'(t) \dd t =
-\int_0^\infty \sum_{i=1}^k \mathcal{R}(e^{A_\Pi t} z_1,..., e^{A_\Pi t}
z_{i-1}, e^{A_\Pi t} A_{\Pi} z_i, e^{A_\Pi t} z_{i+1},..., z_k) \dd t \\
= \ & \sum_{i=1}^k \mathcal{T}(z_1,...,z_{i-1}, A_\Pi z_i, z_{i+1},...,z_k).
\end{align*}
Since $F(0)= \mathcal{R}(z_1,...,z_k)$, equation \eqref{eq:Lyapunov2} is satisfied. Since
$\mD(A^2)$ is dense in $\mD(A)$, equation \eqref{eq:Lyapunov2} remains valid for
$z_i \in \mD(A)$, by continuity.

\emph{Part 2: uniqueness}.
Let $\widetilde{\mT} \in \mM(Y^k,\R)$ satisfy \eqref{eq:Lyapunov2} and let us
set
$\mathcal{E}= \widetilde{\mT}-\mT$. For all $(z_1,...,z_k) \in
\mathcal{D}(A)^k$,
\begin{equation} \label{eq:Lyapounov2d}
\sum_{i=1}^k \mathcal{E}(z_1,...,z_{i-1},A_\Pi z_i, z_{i+1},...,z_k)= 0.
\end{equation}
For a fixed $(z_1,...,z_k) \in \mathcal{D}(A)^k$, we define
\begin{equation*}
G\colon t \in [0,\infty) \mapsto \mathcal{E}(e^{A_\Pi t} z_1,...,e^{A_\Pi t}
z_k).
\end{equation*}
As in the second part of the proof, we can show that $G \in W^{1,1}(0,\infty)$,
with
\begin{equation} \label{eq:Lyapounov2e}
G'(t)= \sum_{i=1}^k \mathcal{E}(e^{A_\Pi t}z_1,..., e^{A_\Pi t} z_{i-1}, A_\Pi
e^{A_\Pi t} z_i, e^{A_\Pi t} z_{i+1}, ..., e^{A_\Pi t} z_k).
\end{equation}
Note that for all $t$, we have that $e^{A_\Pi t} z_i \in
\mathcal{D}(A).$ Hence, we deduce from \eqref{eq:Lyapounov2d} that $G'(t)= 0$
and therefore that $G$ is constant.
For all $i$, we have $\| e^{A_\Pi t}z_i \|_Y \underset{t \to
\infty}{\longrightarrow}
0$, and thus $G(t) \underset{t \to
\infty}{\longrightarrow} 0$ since $\mathcal{E}$ is continuous. This implies that
$G$ is identically 0. Since $G(0)= \mathcal{E}(z_1,...,z_k)$ it follows that
$\mathcal{E}$
is null on $\mathcal{D}(A)^k$. By continuity, $\mathcal{E}$ is null on $Y^k$.
This  concludes the proof.
\end{proof}

\begin{remark} \label{rem:genExistenceUniquenessLyapunov}
Theorem \ref{thm:existenceUniquenessLyapunov} can be generalized to equations
with a right-hand side of the following form:
\begin{align}
\mathcal{R}(z_1,...,z_k)= \ & \mathcal{R}^{(0)}(z_1,...,z_k)
+ \sum_{1 \leq i \leq k}
\mathcal{R}_i^{(1)}(z_1,...,z_{i-1},Nz_i,z_{i+1},...,z_k) \notag \\
& \qquad + \sum_{1 \leq i < j \leq k}
\mathcal{R}_{ij}^{(2)}(z_1,...,z_{i-1},Nz_i,z_{i+1},...,z_{j-1},Nz_j,z_{j+1},...
,z_k), \label{eq:generalizedRHS}
\end{align}
where  $\mathcal{R}^{(0)}$, $(\mathcal{R}_i^{(1)})_{1 \leq i \leq k}$, and
$(\mathcal{R}_{ij}^{(2)})_{1 \leq i < j \leq k}$ are bounded multilinear
forms.
\end{remark}

In the following theorem, we use the nonnegative self-adjoint  Riccati operator
$\Pi$ which was defined in \eqref{eq:algebraicRiccati}.

\begin{theorem} \label{thm:mult_lin_form}
There exists a unique sequence of symmetric multilinear forms $(\mT_k)_{k \geq 2,}$ with $\mT_k \in
\mM(Y^k,\R)$ and a unique sequence of multilinear forms $(\mR_k)_{k \geq 3},$ with $\mR_k \in
\mM(\mD(A)^k,\R)$ such that for all $(z_1,z_2) \in Y^2$,
\begin{equation} \label{eqLyapounov3d}
\mathcal{T}_2(z_1,z_2):= (z_1,\Pi z_2)
\end{equation}
and such that for  all $k \geq 3$, for all $(z_1,...,z_k) \in \mathcal{D}(A)^k$,
\begin{subequations} \label{eqLyapounov3}
\begin{align}
& \sum_{i=1}^k \mathcal{T}_k (z_1,...,z_{i-1},A_\Pi z_i, z_{i+1},...,z_k)=
\mathcal{R}_k(z_1,...,z_k), \label{eqLyapounov3a} \\
& \mathcal{R}_k= 2k(k-1) \text{\emph{Sym}}_{1,k-1}\big( \mathcal{C}_1 \otimes
\mathcal{G}_{k-1} \big)  + \sum_{i=2}^{k-2} \begin{pmatrix} k \\ i \end{pmatrix}
\text{\emph{Sym}}_{i,k-i} \big( (\mathcal{C}_i + i \mathcal{G}_i) \otimes
(\mathcal{C}_{k-i} + (k-i) \mathcal{G}_{k-i}) \big), \label{eqLyapounov3b} \\
& \text{where:} \quad \begin{cases} \begin{array}{rl}
 \mathcal{C}_i(z_1,...,z_i) = \ &
\displaystyle{\mathcal{T}_{i+1}(B,z_1,...,z_i),  \quad \text{for $i=1,...,k-2$,}
}\\
\displaystyle \mathcal{G}_i(z_1,...,z_i) = \ & \displaystyle{\frac{1}{i}  \Big[
\sum_{j=1}^i \mathcal{T}_i(z_1,...,z_{j-1},Nz_j,z_{j+1},...,z_i) \Big], \quad
\text{for $i=1,...,k-1$}.}
\end{array} \end{cases}
\label{eqLyapounov3c}
\end{align}
\end{subequations}
\end{theorem}

\begin{proof}
We prove this claim by induction.  The induction assumption is the following:
for all $p \geq 2$, there exists a unique family $(\mathcal{T}_k)_{2 \leq k \leq
p},$ $\mT_k\in \mM(Y^k,\R)$ and a unique family $(\mR_k)_{3 \leq k \leq p},
\mR_k \in \mM(\mD(A)^k,\R)$ such that
\eqref{eqLyapounov3d} and \eqref{eqLyapounov3} hold, for all $k= 3,...,p$.

For $p=2$, it suffices to  check that $(z_1,z_2) \in Y^2 \mapsto (z_1, \Pi z_2)
\in \R$ is continuous, which directly follows from the Cauchy-Schwarz
inequality and the fact that $\Pi \in \mL(Y)$.

Let $p \geq 2$, assume that the induction assumption is satisfied. Let
$(\mT_k)_{2 \leq k \leq p}, \mT_k\in \mM(Y^k,\R)$ and $(\mR_k)_{3 \leq k \leq
p}, \mR \in \mM(\mD(A)^k,\R)$ be such that \eqref{eqLyapounov3d} and
\eqref{eqLyapounov3} hold, for all $k= 3,...,p$.

Let $\mathcal{R}_{p+1}$ be defined by \eqref{eqLyapounov3b} and
\eqref{eqLyapounov3c} (taking $k= p+1$). The multilinear mapping
$\mathcal{R}_{p+1}$ is well-defined, since \eqref{eqLyapounov3b} and
\eqref{eqLyapounov3c} are defined by
$\mT_2$,...,$\mT_{p}$. Moreover, $\mathcal{R}_{p+1}$ can be
written as a sum of multilinear mappings in which the operator $N$ appears at
most twice. More precisely, since by assumption,
$\mathcal{T}_2$,...,$\mathcal{T}_p$ are bounded, $\mathcal{R}_{p+1}$ can be
written in the form \eqref{eq:generalizedRHS}. Therefore, by Theorem
\ref{thm:existenceUniquenessLyapunov}, there exists a unique $\mT_{p+1}\in
\mM(Y^{p+1},\R)$ satisfying \eqref{eqLyapounov3a}. By induction, $\mathcal{T}_2$,...,$\mathcal{T}_p$ are all symmetric. One can easily check that for $i=1$,...,$p-2$, $\mathcal{C}_i$ is symmetric and for $i=1,...,p-1$, $\mathcal{G}_i$ is symmetric. Therefore, by Lemma \ref{lemma:symmetry}, $\mathcal{R}_{p+1}$ is symmetric and finally, by Theorem \ref{thm:existenceUniquenessLyapunov}, $\mathcal{T}_{p+1}$ is symmetric.
This proves the induction
assumption for $p+1$ and concludes the proof.
\end{proof}

\begin{remark}
In the finite-dimensional case $Y=\mathbb R^n,$ a multilinear form 
$\mathcal{S} \in \mathcal{M}(Y^k,\R)$ can be naturally identified with a 
 {\em multidimensional array} (or {\em tensor}) $\mathbf{S} \in  \mathbb 
R^{n\times \dots \times n}.$ Denoting with $\mathrm{vec}(\mathbf{S})  \in 
\R^{n^k}$ the associated {\em vectorization of  $\mathbf{S}$} allows to 
interpret \eqref{eq:Lyapunov2} as a linear tensor equation of the form
\begin{align*}
 \sum_{i=1}^k (\underbrace{I_n \otimes \dots \otimes I_n}_{i-1} \otimes 
A_{\Pi,n}^T \otimes \underbrace{I_n \otimes \dots \otimes I_n}_{k-i} )  
\mathrm{vec}(\mathbf{T}) = \mathrm{vec}{(\mathbf{R})},
\end{align*}
where $I_n $ is the identity matrix in $\R^{n\times n}$ and $A_{\Pi,n} \in 
\R^{n\times n}$ denotes a finite-dimensional approximation of the operator 
$A_\Pi.$ Let us particularly emphasize that these types of equations can often 
be efficiently solved by tensor methods, see e.g.\@ \cite{Gra04}. 
\end{remark}

For all $p \geq 2$, we define the function $\mathcal{V}_p$ as follows:
\begin{equation}\label{eq:kk20}
  \begin{aligned}
    \mathcal{V}_p & \colon Y \to \R,\quad
    \mathcal{V}_p(y) = \sum_{k=2}^p \frac{1}{k!} \mathcal{T}_k(y,\dots,y),
  \end{aligned}
  \end{equation}
where the sequence $(\mathcal{T}_k)_{k \geq 2}$ is given by Theorem \eqref{thm:mult_lin_form}.
The definition of $\mathcal{V}_p$ is motivated by Theorem \ref{thm:DifferentiabilityImpliesLyapunov}.

\begin{remark}{\em
In Theorem \ref{thm:subOptimality}, we prove that $\mathcal{V}_p$ is an approximation of order $p+1$ of $\mathcal{V}$, in the neighborhood of 0. This result is obtained without assuming the differentiability of
$\mathcal{V}$.}
\end{remark}

\section{Well-posedness of the closed-loop system}
\label{sectionClosedLoop}

In this section, we analyse the non-linear feedback law $\mathbf{u}_p \colon y \in V \rightarrow \R$, defined by
\begin{equation} \label{eq:feedback_law}
\mathbf{u}_p(y) = -
\frac{1}{\alpha} \big( D \mathcal{V}_p(y), Ny+B \big) =
- \frac{1}{\alpha} \Big( \sum_{k=2}^p \frac{1}{(k-1)!} \mathcal{T}_k(N y +
B,y^{\otimes k-1} ) \Big).
\end{equation}
Its form is suggested by \eqref{eqFeedbackGeneral} and \eqref{eq:kk20}.
Note that the explicit expression of $\mathbf{u}_p$ follows from Lemma \ref{lemma:continuousTensors2} and from the symmetry of the multilinear forms $\mathcal{T}_k$.
In this section, we discuss the well-posedness of the closed-loop system
\begin{equation}\label{eq:cl_poly}
\frac{\dd }{\dd t}y = Ay + (Ny+ B) \mathbf{u}_p(y), \quad
y(0) = y_0
\end{equation}
for a fixed value of $p \geq 2$.
We recall that throughout this section and the remainder of the paper, assumptions
\eqref{ass:A1}-\eqref{ass:A4} are supposed to hold.
In Theorem \ref{thm:well_posed_feedback}, we will establish the existence of a
solution to
\eqref{eq:cl_poly}, provided that $\| y_0 \|_Y$ is sufficiently small. We denote
this closed-loop solution by
\begin{equation*}
S(\mathbf{u}_p,y_0).
\end{equation*}
The distinction with the notation $S(u,y_0)$ used for an open-loop control $u
\in L^2(0,\infty;Y)$ will be clear from the context. We also denote by
\begin{equation} \label{eqDefClosedLoop}
\mathbf{U}_p(y_0;t)= \mathbf{u}_p(S(\mathbf{u}_p,y_0;t))
\end{equation}
the open-loop control generated with the feedback law $\mathbf{u}_p$ and the
initial condition $y_0$. We will prove in Corollary \ref{cor:V_loc_bound} that
$\mathbf{U}_p(y_0)$ is well-defined in $L^2(0,\infty)$, provided that $\| y_0
\|_Y$ is small enough.

The strategy that we use to prove the well-posedness of \eqref{eq:cl_poly} is rather standard and has been applied in the context of infinite-dimensional systems several times, see e.g.\@ \cite{BreKP16,Ray06,TheBR10}. It consists in proving that the non-linear part of the closed-loop system satisfies a Lipschitz continuity property. To this purpose, we introduce the nonlinear mapping $F:W_\infty \rightarrow L^2(0,\infty;V^*)$ defined by
\begin{equation}\label{eq:nonlinearity}
F(y)= - \frac{1}{\alpha} (Ny+B) \Big( \sum_{k=3}^p \frac{1}{(k-1)!}
\mathcal{T}_k(Ny+B,y^{\otimes k-1} \Big)
- \frac{1}{\alpha} \Big( Ny \mathcal{T}_2(Ny+B,y) + B \mathcal{T}_2(Ny,y) \Big).
\end{equation}
It can be expressed as the sum of monomial functions of degree
greater or equal to 2.
Observe that the closed-loop system \eqref{eq:cl_poly} can be written as
follows:
\begin{align}
\frac{\dd}{\dd t} y = \ & A y + (Ny+B)\mathbf{u}_p(y) \notag \\
= \ & A y + (Ny+B)\Big(-\frac{1}{\alpha} \sum_{k=2}^p \frac{1}{(k-1)!} \mT_k(Ny+B,y^{\otimes k-1}) \Big) \notag \\
= \ & \Big( A - \frac{1}{\alpha}BB^* \Pi \Big) y + F(y)
=  A_\Pi y + F(y). \label{eq:nonlinearities}
\end{align}
In Lemma \eqref{lem:LipschitzNonLin} we prove that $F$ is well-defined and Lipschitz
continuous on bounded subsets (for the $L^\infty(0,\infty;Y)$-norm), and that
the associated Lipschitz modulus can be made as small as necessary, by
restricting the size of the considered subset. The well-posedness of \eqref{eq:cl_poly} is then obtained in Theorem \eqref{thm:well_posed_feedback} with a fixed-point argument.

We set
\begin{align*}
  W_\infty:= W(0,\infty) = \left\{y \in L^2(0,\infty;V): \frac{\dd
}{\dd t}y \in L^2(0,\infty;V^*)\right \}.
\end{align*}
We recall that $W_\infty$ is continuously embedded in $C(0,\infty;Y)$
\cite[Theorem 3.1]{LioM72}: there
exists a constant $C_0>0$ such that for all $y \in W_\infty$,
\begin{equation} \label{eq:WinC}
\| y \|_{L^\infty(0,\infty;Y)} \leq C_0 \| y \|_{W_\infty}.
\end{equation}

The following lemma is a technical lemma, used for analysing the non-linear mapping $F$.

\begin{lemma} \label{lem:Lipschitz1}
There exists a constant $C>0$ such that for all $\delta \in [0,1]$, for all $k=2,...,p$, and for all
$y_1$ and $y_2 \in B_Y(\delta) \cap V$,
\begin{align*}
& \| (Ny_2+B)\mathcal{T}_k(Ny_2+B,y_2^{\otimes k-1})
- (Ny_1+B)\mathcal{T}_k(Ny_1+B,y_1^{\otimes k-1}) \|_{V^*} \\
& \qquad \leq C \big( \delta \| y_2-y_1 \|_V + (\| y_1 \|_V + \| y_2 \|_V )
\  \| y_2-y_1 \|_Y \big).
\end{align*}
\end{lemma}

\begin{proof}
Let $\delta \in [0,1]$, let $y_1$ and $y_2 \in B_Y(\delta)$. Then we have
\begin{align*}
& \| (Ny_2+B)\mathcal{T}_k(Ny_2+B,y_2^{\otimes k-1})
- (Ny_1+B)\mathcal{T}_k(Ny_1+B,y_1^{\otimes k-1}) \|_{V^*} \\
& \qquad \leq  \underbrace{\| N(y_2-y_1) \mathcal{T}_k(Ny_2+B, y_2^{\otimes k-1})
\|_{V^*}}_{=:(a)}
+ \underbrace{\| (Ny_1+B) \mathcal{T}_k(N(y_2-y_1),y_2^{\otimes k-1})
\|_{V^*}}_{=:(b)} \\
& \qquad \qquad + \underbrace{\| (Ny_1+B) \big( \mathcal{T}_k(Ny_1+B,y_2^{\otimes k-1}) -
\mathcal{T}_k(Ny_1+B,y_1^{\otimes k-1}) \big) \|_{V^*}}_{=:(c)}.
\end{align*}
We need to find a bound on  the three terms of the right-hand side of the
above inequality. Note first that $\| Ny_1 + B \|_{V^*} \leq M:= \| N
\|_{\mathcal{L}(Y,V^*)} + \| B \|_Y$.
We have
\begin{align*}
(a) \leq \ & \| N \|_{\mathcal{L}(Y,V^*)} \  \| y_2-y_1 \|_Y \  \|
\mathcal{T}_k \| \  \big( \| N \|_{\mathcal{L}(V,Y)} \  \| y_2 \|_V + \| B
\|_Y  \big) \delta^{k-1}, \\
(b) \leq \ & M \  \| \mathcal{T}_k \| \  \| N \|_{\mathcal{L}(V,Y)}
\  \| y_2-y_1 \|_V \  \delta^{k-1}, \\
(c) \leq \ & M (k-1) \delta^{k-2} \  \| \mathcal{T}_k \| \
\big( \| N \|_{\mathcal{L}(V,Y)} \  \| y_1 \|_V + \| B \|_Y  \big) \
\|y_2-y_1 \|_Y.
\end{align*}
For the upper estimate of $(c)$, we used Lemma \ref{lemma:continuousTensors1}
and the fact that
\begin{equation*}
\| \mathcal{T}_k(Ny_1+B,\cdot,...,\cdot) \| \leq \big( \| N
\|_{\mathcal{L}(V,Y)} \cdot \| y_1 \|_V + \| B \|_Y  \big) \cdot \|
\mathcal{T}_k \|.
\end{equation*}
The lemma follows, since $\delta^{k-1} \leq \delta$ and since $V$ is
continuously embedded in $Y$.
\end{proof}

We now prove a Lipschitz continuity property satisfied by $F$.

\begin{lemma} \label{lem:LipschitzNonLin}
The mapping $F$ is well-defined.
Moreover, there exists a constant $C_1>0$ such that for all $\delta \in [0,1]$,
for all $y_1$ and $y_2 \in W_\infty$ with $\| y_1 \|_{L^\infty(0,\infty;Y)} \leq
\delta$ and $\| y_2 \|_{L^\infty(0,\infty;Y)} \leq \delta$,
\begin{equation} \label{eq:LipschitzNonLin1}
\| F(y_2)-F(y_1) \|_{L^2(0,\infty;V^*)} \leq C_1 \big( \delta + \| y_1
\|_{L^2(0,\infty;V)} + \| y_2 \|_{L^2(0,\infty;V)} \big) \| y_2-y_1
\|_{W_\infty}.
\end{equation}
\end{lemma}

\begin{proof}
Observe that $F(0)= 0$. Therefore, \eqref{eq:LipschitzNonLin1} will ensure that
$F(y) \in L^2(0,\infty;V^*)$ (at least for $\| y \|_{L^\infty(0,\infty;Y)} \leq
1$, but the well-posedness can actually be checked for any $y$). Let $y_1$ and $y_2 \in
W_\infty$ be such that $\| y_1 \|_{L^\infty(0,\infty;Y)} \leq \delta$ and $\|
y_2 \|_{L^\infty(0,\infty;Y)} \leq \delta$.
By Lemma \ref{lem:Lipschitz1},
\begin{align*}
& \big\| \big[ \big( Ny_2(\cdot)+B \big) \mathcal{T}_k \big(
Ny_2(\cdot)+B,y_2^{\otimes k-1}(\cdot) \big) \big]
- \big[ \big( Ny_1(\cdot)+B \big) \mathcal{T}_k \big( Ny_1(\cdot)+B,y_1^{\otimes
k-1}(\cdot) \big) \big] \big\|_{L^2(0,\infty;V^*)} \\
& \qquad \leq C \big( \delta \| y_2 - y_1 \|_{L^2(0,\infty;V)} + (\| y_1
\|_{L^2(0,\infty;V)} + \| y_2 \|_{L^2(0,\infty;V)} ) \| y_2 - y_1
\|_{L^\infty(0,\infty;Y)} \big).
\end{align*}
With estimates similar to the ones used in Lemma \ref{lem:Lipschitz1}, one can
show that
\begin{align*}
& \big\| \big[ Ny_2(\cdot) \mathcal{T}_2 \big( Ny_2(\cdot)+B,y_2(\cdot) \big) +
B \mathcal{T}_2 \big( Ny_2(\cdot),y_2(\cdot) \big) \big] \\
& \quad - \big[ Ny_1(\cdot) \mathcal{T}_2 \big( Ny_1(\cdot)+B,y_1(\cdot) \big) +
B \mathcal{T}_2 \big( Ny_1(\cdot),y_1(\cdot) \big)
\big]
\big \|_{L^2(0,\infty;V^*)} \\
& \qquad \qquad \leq C \big( \delta \| y_2 - y_1 \|_{L^2(0,\infty;V)} + (\| y_1
\|_{L^2(0,\infty;V)} + \| y_2 \|_{L^2(0,\infty;V)} ) \| y_2 - y_1
\|_{L^\infty(0,\infty;Y)} \big).
\end{align*}
Using the continuous embedding of $W_\infty$ in $L^\infty(0,\infty;Y)$, we
obtain \eqref{eq:LipschitzNonLin1}, which concludes the proof.
\end{proof}

With regard to a fixed-point argument, let us consider the linearized
nonhomogeneous system associated to \eqref{eq:cl_poly}
\begin{equation}\label{eq:nonhomg}
 \frac{\dd }{\dd t} z = A_\Pi z  + f, \quad z(0) =y_0
\end{equation}
for which we have the following result.

\begin{proposition}
\label{prop:reg_nonh}
There exists a constant $C_2>0$ such that for all $f \in L^2(0,\infty;V^*)$ and
for all $y_0 \in Y$, there exists a unique mild solution $z \in W_\infty$ to
\eqref{eq:nonhomg} satisfying
\begin{equation*}
 \| z  \| _{W_\infty} \le C_2 ( \| f\|_{L^2(0,\infty; V^*)} + \| y_0
\|_Y ) .
\end{equation*}
In particular, $z \in C_b([0,\infty);Y)$.
\end{proposition}

This result can be verified with the techniques of \cite[Theorem 2.2, Part II,
Chapter 3]{Benetal07} and \cite{TheBR10}. We are now ready to prove the well-posedness of \eqref{eq:cl_poly}.

\begin{theorem} \label{thm:well_posed_feedback}
There exist two constants $\delta_0>0$ and $C>0$
such that for all $y_0 \in B_Y(\delta_0)$, the closed-loop system
\eqref{eq:cl_poly} admits a unique solution $S(\mathbf{u}_p,y_0) \in W_\infty$ satisfying
\begin{equation}
\| S(\mathbf{u}_p,y_0) \| _{W_\infty} \leq C \| y_0 \|_Y.
\end{equation}
Moreover, the mapping $y_0 \in B_Y(\delta_0) \mapsto S(\mathbf{u}_p,y_0)$ is Lipschitz continuous.
\end{theorem}

\begin{proof}
In the proof, we denote by $C_0$ the constant involved in \eqref{eq:WinC} and
by $C_1$ and $C_2$ the two constants obtained in Lemma \ref{lem:Lipschitz1} and
Lemma \ref{lem:LipschitzNonLin}.
We set
\begin{equation*}
C= C_2+1, \quad
\delta_0= \min\Big(\frac{1}{C C_0}, \frac{1}{2 C^2(C_0 + 1) C_1 C_2},
\frac{1}{2C (C_0 +2) C_1 C_2} \Big).
\end{equation*}
Let us fix $y_0 \in B_Y(\delta_0)$.
Consider the mapping $\mathcal{Z} \colon y \in W_\infty \mapsto \mathcal{Z}(y)$,
where $\mathcal{Z}(y)$ is the solution of
\begin{equation*}
\frac{\dd }{\dd t} z = A_\Pi z + F(y), \ \ z(0)=y_0,
\end{equation*}
which exists by Proposition \ref{prop:reg_nonh}.
We show that $\mathcal{Z}$ is a contraction in
\begin{equation*}
\Omega: = \big\{ y \in W_\infty \colon \| y \|_{W_\infty} \leq C \| y_0 \|_Y
\big\}.
\end{equation*}
Note that $\| y \|_{W_\infty} \leq C \| y_0 \|_Y \leq C \delta_0$ for all $y \in
\Omega$ and that
\begin{equation} \label{eq:wellPosed1}
\| y \|_{L^\infty(0,\infty;Y)}
\leq C_0 \| y \|_{W_\infty}
\leq C_0 C \delta_0
\leq 1.
\end{equation}
Let us show that $\mathcal{Z}(\Omega) \subseteq \Omega$. Let $y \in \Omega$.
Applying Lemma \ref{lem:LipschitzNonLin} (with $\delta= C_0 C \delta_0$), we
obtain
\begin{align*}
& \| F(y) \|_{L^2(0,\infty;V^*)}
= \| F(y) - F(0) \|_{L^2(0,\infty;V^*)}
\leq C_1 (\delta+ C \delta_0) \| y \|_{W_\infty} \\
& \qquad \leq C_1(C_0 C \delta_0 + C \delta_0) C \| y_0 \|_Y
\leq C^2 (C_0+1) C_1 \delta_0 \| y_0 \|_Y.
\end{align*}
Therefore, by Proposition \ref{prop:reg_nonh},
\begin{align*}
& \| \mathcal{Z}(y) \|_{W_\infty}
\leq C_2 \big( \| F(y) \|_{L^2(0,\infty;V^*)} + \| y_0 \|_Y \big) \\
& \qquad \leq \underbrace{C^2 (C_0+1) C_1 C_2 \delta_0}_{\leq 1} \| y_0 \|_Y^2 +
C_2 \| y_0
\|_Y
\leq (C_2 +1) \| y_0 \|_Y,
\end{align*}
which proves that $\mathcal{Z}(y) \in \Omega$.

Next, for $y_1$ and $y_2 \in \Omega$ we set $z=
\mathcal{Z}(y_2)-\mathcal{Z}(y_1)$. Then we have
\begin{equation*}
\frac{\dd}{\dd t} z = A_\Pi z + F(y_2)-F(y_1), \quad z(0)= 0.
\end{equation*}
Taking $\delta= C_0 C \delta_0$ and applying Lemma
\ref{lem:LipschitzNonLin} and Propositon \ref{prop:reg_nonh}, we obtain
\begin{align*}
& \| \mathcal{Z}(y_2)-\mathcal{Z}(y_1) \|_{W_\infty} = \| z \|_{W_\infty}
\leq C_2 \| F(y_2)- F(y_1) \|_{L^2(0,\infty;V^*)} \\
& \qquad \leq C_2 C_1 \big( \delta + \underbrace{\| y_1
\|_{L^2(0,\infty;V)}}_{\leq C \delta_0} + \underbrace{\| y_2
\|_{L^2(0,\infty;V)}}_{\leq C \delta_0} \big) \| y_2 - y_1 \|_{W_\infty}
\leq \underbrace{C(C_0+2)C_1 C_2 \delta_0}_{\leq 1/2} \| y_2 - y_1
\|_{W_\infty}.
\end{align*}
Hence, $\mathcal{Z}$ is a contraction and the well-posedness of
\eqref{eq:cl_poly} follows with the Banach fixed point theorem.

We finally prove that the mapping $y_0 \in B_Y(\delta_0) \mapsto
S(\mathbf{u}_p,y_0)$ is Lipschitz continuous. Let $y_{1,0}$ and $y_{2,0} \in
B_Y(\delta_0)$, let $y_1= S(\mathbf{u}_p,y_{1,0})$, let $y_2=
S(\mathbf{u}_p,y_{2,0})$, let $z= y_2-y_1$. It holds
\begin{equation*}
\frac{\dd}{\dd t}z= A_{\Pi} z + F(y_2) - F(y_1), \quad
z(0)= y_{2,0}-y_{1,0}.
\end{equation*}
By \eqref{eq:wellPosed1}, we obtain
\begin{equation*}
\| y_1 \|_{L^\infty(0,\infty;Y)} \leq C_0 C \delta_0 \quad \text{and} \quad
\| y_2 \|_{L^\infty(0,\infty;Y)} \leq C_0 C \delta_0.
\end{equation*}
Applying again Lemma \ref{lem:LipschitzNonLin} with $\delta= C_0 C \delta_0$, we
obtain that
\begin{equation*}
\| F(y_2) - F(y_1) \|_{L^2(0,\infty;V^*)}
\leq C(C_0 +2) C_1 \delta_0 \| y_2 - y_1 \|_{W_\infty}.
\end{equation*}
Therefore, by Proposition \ref{prop:reg_nonh},
\begin{align*}
\| y_2-y_1 \|_{W_\infty}
\leq \ & C_2 \| F(y_2)-F(y_1) \|_{L^2(0,\infty;V^*)} + C_2 \| y_{2,0}-y_{1,0}
\|_Y \\
\leq \ & \underbrace{C(C_0 +2) C_1 C_2 \delta_0 }_{\leq 1/2} \| y_2 - y_1
\|_{W_\infty} + C_2 \| y_{2,0} - y_{1,0} \|_Y.
\end{align*}
It follows that
\begin{equation*}
\| y_2-y_1 \|_{W_\infty} \leq 2 C_2 \| y_{2,0}-y_{1,0} \|_Y,
\end{equation*}
which proves the Lipschitz continuity of the mapping $y_0 \mapsto
S(\mathbf{u}_p,y_0)$ and concludes the proof of the theorem.
\end{proof}

\begin{corollary} \label{cor:V_loc_bound}
Let $\delta_0$ be given by Theorem \ref{thm:well_posed_feedback}.
The following mapping:
\begin{equation*}
y_0 \in B_Y(\delta_0) \mapsto \mathbf{U}_p(y_0)=
\mathbf{u}_p(S(\mathbf{u}_p,y_0;\cdot)) \in L^2(0,\infty)
\end{equation*}
is well-defined and continuous. Moreover, there exists a constant $C>0$ such
that for all $y_0 \in B_Y(\delta_0)$,
\begin{equation} \label{eq:V_loc_bound}
\mathcal{V}(y_0) \leq \mathcal{J}(\mathbf{U}_p(y_0),y_0) \leq C \|y_0\|_Y^2.
\end{equation}
\end{corollary}

\begin{proof}
We begin by proving the well-posedness and the continuity of $\mathbf{U}_p$. We
actually prove that the mapping is Lipschitz continuous. Since
$\mathbf{U}_p(0)=0$, the Lipschitz continuity ensures also the well-posedness.
We set $\Omega= S(\mathbf{u}_p,B_Y(\delta_0)) \subset W_\infty$. By Theorem
\ref{thm:well_posed_feedback}, there exists $\delta > 0$ such that for all $y
\in \Omega$,
\begin{equation*}
\| y \|_{L^\infty(0,\infty;Y)} \leq \delta \quad \text{and} \quad
\| y \|_{L^2(0,\infty;Y)} \leq \delta.
\end{equation*}
For all $y_1$ and $y_2 \in B_Y(\delta)$,
\begin{align*}
& \big| \mathcal{T}_k(Ny_2 + B, y_2^{\otimes k-1}) - \mathcal{T}_k(Ny_1 +B,
y_1^{\otimes k-1}) \big| \\
& \qquad \leq \big| \mathcal{T}_k(N(y_2-y_1),y_2^{\otimes k-1}) \big|
+ \big| \mathcal{T}_k(Ny_1 +B,y_2^{\otimes k-1})-\mathcal{T}_k(Ny_1 +B, y_1^{
\otimes k-1}) \big| \\
& \qquad \leq \| \mathcal{T}_k \| \  \| N \|_{\mathcal{L}(V,Y)} \  \| y_2 -y_1
\|_V \  \delta^{k-1} \\
& \qquad \qquad + \| \mathcal{T}_k \| \big( \| N \|_{\mathcal{L}(V,Y)} \   \| y_1
\|_V + \| B \|_Y \big)
(k-1) \delta^{k-2} \| y_2-y_1 \|_Y.
\end{align*}
In the last inequality, we used Lemma \ref{lemma:continuousTensors1} and the
fact that
\begin{equation*}
\| \mathcal{T}_k(Ny_1 + B, \cdot,...,\cdot) \| \leq \| \mathcal{T}_k \| \
\big( \| N \|_{\mathcal{L}(V,Y)} \ \| y_1 \|_V + \| B \|_Y \big).
\end{equation*}
As a consequence, for all $y_1$ and $y_2 \in \Omega$,
\begin{align*}
& \big\| \mathcal{T}_k \big( Ny_2(\cdot) + B, y_2^{\otimes k-1}(\cdot) \big) -
\mathcal{T}_k \big( Ny_1(\cdot) + B, y_1^{\otimes  k-1}(\cdot) \big)
\big\|_{L^2(0,\infty)}^2 \\
& \qquad \leq C \big( \| y_2- y_1 \|_{L^2(0,\infty;V)}^2 + \| y_1
\|_{L^2(0,\infty;V)}^2 \| y_2 - y_1 \|_{L^\infty(0,\infty;Y)}^2 + \| y_2 - y_1
\|_{L^2(0,\infty;Y)}^2 \big) \\
& \qquad \leq C \| y_2-y_1 \|^2_{W_\infty}.
\end{align*}
It follows that the mapping: $y \in \Omega \mapsto \mathbf{u}_p(y(\cdot)) \in
W_\infty$ is
Lipschitz continuous. By composition with $y_0 \in B_Y(\delta_0) \mapsto
S(\mathbf{u}_p,y_0)$, the mapping $\mathbf{U}_p$ is Lipschitz continuous and
well-posed.

Let us prove inequality \eqref{eq:V_loc_bound}. Since $S(\mathbf{u}_p,\cdot)$
and $\mathbf{U}_p$ are both Lipschitz continuous, there exists $C>0$ such that
for all $y_0 \in B_Y(\delta_0)$,
\begin{equation*}
\| S(\mathbf{u}_p,y_0) \|_{L^2(0,\infty;Y)} \leq C \| y_0 \|_Y \quad \text{and}
\quad
\| \mathbf{U}_p(y_0) \|_{L^2(0,\infty;Y)} \leq C \| y_0 \|_Y.
\end{equation*}
It follows that
\begin{equation*}
\mathcal{V}(y_0) \leq \mathcal{J}(\mathbf{U}_p(y_0), y_0) \leq C^2(1+ \alpha)/2
\  \| y_0 \|_Y^2,
\end{equation*}
which concludes the proof.
\end{proof}

\section{Properties of the optimal control}
\label{sectionOptimalOpenLoopControl}

\begin{proposition} \label{prop:bound_opt_sol}
Let $\delta_0>0$ be given by Theorem \ref{thm:well_posed_feedback}. Then, for all
$y_0
\in B_Y(\delta_0)$, there exists a solution $u$ to problem
\eqref{eqProblem} with initial value $y_0$.
Moreover, $y:= S(u,y_0)$ lies in $L^2(0,\infty;V) \cap L^\infty(0,\infty; Y)$
and the following estimates hold:
\begin{equation} \label{eq:estimateForBarY}
\| y \|_{L^\infty(0,\infty;Y)} \leq C \| y_0 \|_{Y}
\quad \text{and} \quad
\| y \|_{L^2(0,\infty;V)} \leq C \| y_0 \|_{Y},
\end{equation}
where the constant $C$ is independent of $y_0$.
\end{proposition}

\begin{proof}
By Corollary \ref{cor:V_loc_bound}, we have $\mathcal{V}(y_0) \leq C \| y_0
\|_Y^2 \leq C \delta_0^2$.
Hence, Proposition \ref{prop:kk1} guarantees the existence of a
solution $u$ to problem \eqref{eqProblem}, with initial condition $y_0$. Let $y= S(u,y_0)$.
We deduce from $\mathcal{V}(y_0)= \| y \|_{L^2(0,\infty;Y)}^2 + \frac{\alpha}{2} \| u \|_{L^2(0,\infty)}^2$ that
\begin{equation} \label{eq:ExistenceSol3}
\| u \|_{L^2(0,\infty)}^2
\leq \frac{2}{\alpha} C \| y_0 \|^2 \leq \frac{2C \delta_0^2}{\alpha}
\quad \text{and} \quad
\| y \|_{L^2(0,\infty;Y)}^2
\leq C \| y_0 \|_Y^2
\leq 2C \delta_0^2.
\end{equation}
Estimate \eqref{eq:estimateForBarY} follows then from
\eqref{eq:RegEstimBis1}, \eqref{eq:RegEstimBis2}, and \eqref{eq:ExistenceSol3}.
\end{proof}

\begin{proposition} \label{prop:continuityValueFunction}
The value function $\mathcal{V}$ is continuous on $B_Y(\delta_0)$,  with
$\delta_0 > 0$ given by Theorem \ref{thm:well_posed_feedback}.
\end{proposition}

\begin{proof}
Let $\varepsilon_2 > 0$. We construct $\varepsilon_1 > 0$ in such a way that for
all $\hat{y}_0 \in B_Y(\delta_0)$ and $\tilde{y}_0 \in B_Y(\delta_0)$,
\begin{equation} \label{eq:Continuity}
\| \tilde{y}_0-\hat{y}_0 \|_Y \leq \varepsilon_1 \Longrightarrow
| \mathcal{V}(\tilde{y}_0)-\mathcal{V}(\hat{y}_0) | \leq \varepsilon_2.
\end{equation}
Before defining $\varepsilon_1$, we need to introduce some constants.
By Corollary \ref{cor:V_loc_bound}, there exists a constant $C>0$ such that for
all $y_0 \in B_Y(\delta_0)$, $\mathcal{V}(y_0) \leq C \| y_0 \|_Y^2 \leq C
\delta_0^2$.
We set
\begin{equation*}
\varepsilon_3= \frac{1}{2} \min \Big[  \delta_0, \Big( \frac{\varepsilon_2}{2C}
\Big)^{1/2} \, \Big] \quad
\text{and} \quad
T= \frac{C \delta_0^2}{\varepsilon_3^2}.
\end{equation*}
The constant $T$ is defined in such a way that
for each  solution $u$ to \eqref{eqProblem} with initial value $y_0 \in
B_Y(\delta_0)$, there exists
$\tau \in [0,T]$ such that $\| S(u,y_0;\tau) \|_Y \leq \varepsilon_3$. Indeed,
if it was not the case, one would have
\begin{equation*}
\mathcal{V}(y_0) > \int_0^{T} \| S(u,y_0;t) \|_Y^2 \dd t \geq T \varepsilon_3^2=
C \delta_0^2,
\end{equation*}
in contradiction with Corollary \ref{cor:V_loc_bound}.
For all $y_0 \in B_Y(\delta_0)$, it holds: $\mathcal{V}(y_0) \leq C \delta_0^2$,
and therefore, if $u$ is an optimal solution to \eqref{eqProblem} with initial
value $y_0$, then
\begin{equation}\label{eq:aux1}
\| u \|_{L^2(0,\infty)}^2 \leq \frac{2}{\alpha} C \delta_0^2.
\end{equation}
By Lemma \ref{lemma:RegEstim}, there exist $M$ and $L>0$ such that for all $u
\in L^2(0,T)$ with $\| u \|_{L^2(0,T)}^2 \leq 2C\delta^2/\alpha$, for all $y_0$
and $\tilde{y}_0 \in B_Y(\delta_0)$,
\begin{equation}\label{eq:aux2}
\| S(u,y_0) \|_{L^\infty(0,T;Y)} \leq M \quad \text{and} \quad
\| S(u,\tilde{y}_0)-S(u,y_0) \|_{L^\infty(0,T;Y)} \leq L \| \tilde{y}_0 - y_0
\|_Y.
\end{equation}
We finally define
\begin{equation*}
\varepsilon_1= \min \Big( \frac{\varepsilon_2}{4TML}, \frac{\varepsilon_3}{L}
\Big).
\end{equation*}
We are ready to prove \eqref{eq:Continuity}. Let $\tilde{y}_0$ and $\hat{y}_0
\in B_Y(\delta_0)$ be such that $\| \tilde{y}_0-\hat{y}_0 \|_Y \leq
\varepsilon_1$. Let $\tilde{u}$ and $\hat{u}$ be associated optimal solutions,
and let $\tilde{y}$ and $\hat{y}$ be the associated trajectories.
Take $\tau \in [0,T]$ such that $\| \tilde{y}(\tau) \|_Y \leq \varepsilon_3$.
By \eqref{eq:aux1} and \eqref{eq:aux2}, we have
\begin{equation*}
\| S(\tilde{u},\hat{y}_0) - S(\tilde{u},\tilde{y}_0) \|_{L^\infty(0,T;Y)} \leq L
\| \hat{y}_0-\tilde{y}_0 \|_Y \leq L \varepsilon_1.
\end{equation*}
We set $y_1= S(\tilde{u},\hat{y}_0;\tau)$. It holds that
\begin{equation*}
\| y_1 \|_Y
\leq \|  S(\tilde{u},\tilde{y}_0) \|_Y +  \| S(\tilde{u},\hat{y}_0) -
S(\tilde{u},\tilde{y}_0) \|_Y \leq \underbrace{\| \tilde{y}(\tau) \|_Y}_{\leq \varepsilon_3} + L \varepsilon_1
\leq 2 \varepsilon_3.
\end{equation*}
Therefore, using the definition of $\varepsilon_3$, we obtain that $\| y_1 \|_Y
\leq \delta_0$ and thus that
\begin{equation} \label{eq:Continuity22}
\mathcal{V}(y_1) \leq C (2 \varepsilon_3)^2 \leq \varepsilon_2/2.
\end{equation}
By the dynamic programming principle, see \eqref{eq2.3}, we have
\begin{equation} \label{eq:Continuity1}
\mathcal{V}(\hat{y}_0) \leq \int_0^{\tau} \ell \big( S(\tilde{u},\hat{y}_0;t),\tilde{u}(t) \big) \dd t + \mathcal{V}(
\underbrace{S(\tilde{u},y_0;\tau)}_{=y_1})
\leq \int_0^{\tau} \ell \big( S(\tilde{u},\hat{y}_0;t),\tilde{u}(t) \big) \dd t + \varepsilon_2/2.
\end{equation}
We find now an upper estimate on the integral of the r.h.s.\@ in the above inequality.
We have
\begin{align}
\int_0^{\tau} \ell \big( S(u,\hat{y}_0;t),u(t) \big) & \dd t
=  \int_0^{\tau} \ell \big( S(\tilde{u},\hat{y}_0,;t),\tilde{u}(t) \big) \dd t
\notag \\
\leq \ & \frac{1}{2} \int_0^{\tau} \| S(\tilde{u},\tilde{y}_0;t) \|_Y^2 + \alpha
|\tilde{u}(t)|^2
+ \| S(\tilde{u},\hat{y}_0;t) \|_Y^2- \| S(\tilde{u},\tilde{y}_0;t) \|_Y^2
\notag \\
\leq \ & \mathcal{J}(\tilde{u},\tilde{y}_0) + \frac{1}{2} \int_0^{\tau} \langle
S(\tilde{u},\hat{y}_0;t)-S(\tilde{u},\tilde{y}_0;t), S(\tilde{u},\hat{y}_0;t) +
S(\tilde{u},\tilde{y}_0;t) \rangle_Y \dd t \notag \\
\leq \ & \mathcal{V}(\tilde{y}_0) + \frac{1}{2} T L \varepsilon_1 \ 2M
\leq \  \mathcal{V}(\tilde{y}_0) + \varepsilon_2/2. \label{eq:Continuity2}
\end{align}
Combining \eqref{eq:Continuity1} and \eqref{eq:Continuity2}, we obtain that $\mathcal{V}(\hat{y}_0) \leq
\mathcal{V}(\tilde{y}_0) + \varepsilon_2$. One can similarly prove that
$\mathcal{V}(\tilde{y}_0) \leq \mathcal{V}(\hat{y}_0) + \varepsilon_2$, by
exchanging the symbols ``$\ \hat{} \ $" and ``$\ \tilde{} \ $" in the above
proof. Therefore, \eqref{eq:Continuity} holds and the continuity of
$\mathcal{V}$ is demonstrated.
\end{proof}

\section{Error estimate for the polynomial approximation}
\label{sectionTaylorExpandability}

In this section, we prove the two main results of the article.
In Theorem \ref{thm:subOptimality}, we give an estimate for the
quality of the feedback law $\mathbf{u}_p$ for $\| y_0 \|_{Y}$ small
enough. This will be based on the fact that $\mathcal{V}_p$ provides a Taylor
approximation of $\mathcal{V}$ of order $p+1$  in a neighborhood in $Y$ of 0.
In Theorem \ref{thm:errorEstimForUp}, we give an estimate for $\| \bar{u} - \mathbf{U}_p(y_0) \|_{L^2(0,\infty)}$, where $\bar{u}$ is a solution to problem \eqref{eqProblem} with initial condition $y_0$, with $y_0$ small enough, and where $\mathbf{U}_p(y_0)$ is the open-loop control associated with the feedback law $\mathbf{u}_p$ and the initial condition $y_0$ (see the definition given by \eqref{eqDefClosedLoop}).

Our analysis consists first in defining a perturbed cost function
$\mathcal{J}_p$ which has the property that $\mathcal{V}_p$ is its minimal value
functional over a set of controls specified below.
This is achieved by constructing a remainder term $r_p$,
defined for $p \geq 2$ and $y \in V$ by
\begin{equation} \label{eq:poly_val_func}
r_p(y)= \frac{1}{2\alpha} \sum_{i=p+1}^{2p} \sum_{j=i-p}^p
q_{p,j}(y)q_{p,i-j}(y),
\quad \text{where: }
\begin{cases}
\begin{array}{l}
q_{p,1}(y)= \mathcal{C}_1(y), \\
q_{p,i}(y)= \frac{1}{i!} \big( \mathcal{C}_i(y^{\otimes i}) + i
\mathcal{G}_i(y^{\otimes i}) \big), \\\qquad \qquad \qquad \qquad \text{$\forall
i=2,...,p-1$,} \\
q_{p,p}(y)= \frac{1}{(p-1)!} \mathcal{G}_p(y^{\otimes p}).
\end{array}
\end{cases}
\end{equation}
We recall that the definitions of $\mathcal{C}_i$ and $\mathcal{G}_i$ are given by \eqref{eqLyapounov3c}.
The perturbed cost function $\mathcal{J}_p$ is defined by
\begin{equation*}
\mathcal{J}_p(u,y_0):= \frac{1}{2} \int_0^\infty \| S(u,y_0;t) \|_Y^2 \dd t +
\frac{\alpha}{2} \int_0^\infty u^2(t) \dd t + \int_0^\infty r_p \big( S(u,y_0;t)
\big) \dd t.
\end{equation*}
The well-posedness of $\mathcal{J}_p$, for a certain class of controls, will be investigated in Lemma \ref{lemma:optimalityJp}. Note that $r_p$ is not necessarily non-negative.

\begin{proposition}\label{prop:resi_hjb}
For all $p \geq 2$ and all $y \in \mathcal{D}(A)$, we have
\begin{equation} \label{eq:resi_hjb1}
r_p(y)= - D\mathcal{V}_p(y)(Ay) -
\frac{1}{2} \| y \|_Y^2 + \frac{1}{2\alpha }
\big( D\mathcal{V}_p(y) (Ny + B)
\big)^2.
\end{equation}
Moreover, for all $p \geq 2$, there exists a constant $C>0$ such that for all $y
\in V$,
\begin{equation} \label{eq:resi_hjb2}
|r_p(y)| \leq C \sum_{i=p+1}^{2p} \| y \|_V^2 \, \| y \|_{Y}^{i-2}.
\end{equation}
\end{proposition}

\begin{remark} \label{rem:HJBforVp}
{\em Relation \eqref{eq:resi_hjb1} states that $\mathcal{V}_p$ is a
solution to the HJB equation associated with the problem of minimizing
$\mathcal{J}_p$ (compare with Proposition \eqref{prop:HJB}). This relation is the key tool to
establish Lemma
\ref{lemma:optimalityJp}.
}
\end{remark}

\begin{proof}[Proof of Proposition \ref{prop:resi_hjb}]
Let us prove \eqref{eq:resi_hjb1}. Let us fix $y \in \mathcal{D}(A)$.
For $p=2$, using that the operator $\Pi$ generating $\mathcal{T}_2$ is the
solution of the algebraic
Riccati equation, we obtain
\begin{align*}
& -  D\mathcal{V}_p(y)(Ay) -
\frac{1}{2} \| y \|_Y^2 + \frac{1}{2\alpha }
\left( D\mathcal{V}_p(y) (Ny + B)
\right)^2 \\
& \qquad = -\mathcal{T}_2(Ay,y)  - \frac{1}{2}\| y \|_Y^2  +
\frac{1}{2\alpha}(\mathcal{T}_2(Ny+B,
y))^2 \\
& \qquad = \underbrace{-\mathcal{T}_2(Ay,y) - \frac{1}{2} \| y \|_Y^2 +
\frac{1}{2\alpha} \mathcal{T}_2(B,y)^2 }_{= 0} +
\frac{1}{2\alpha}
\big[
2\mathcal{T}_2(B,y) \mathcal{T}_2(Ny,y)
+
\mathcal{T}_2(Ny,y)^2 \Big] \\
& \qquad = \frac{1}{2\alpha} \Big[ 2 q_{2,1}(y)q_{2,2}(y) + q_{2,2}(y)^2 \Big]
=  \frac{1}{2\alpha} r_2(y).
\end{align*}

Now let $p \geq 3$. Our proof is based on Theorem \ref{thm:mult_lin_form}. The expressions of the multilinear forms $\mathcal{C}_i$, $\mathcal{G}_i$, and $\mathcal{R}_k$ can be simplified when the mappings are evaluated at $y^{\otimes i}$ and $y^{\otimes k}$, respectively.
By definition of $\mathcal{C}_i$ and $\mathcal{G}_i$ (see \eqref{eqLyapounov3c}) and using the symmetry of the multilinear forms $\mathcal{T}_i$ (proved in Theorem \ref{thm:mult_lin_form}),
\begin{equation} \label{eq:rpProof1}
\mathcal{C}_i(y^{\otimes i})= \mathcal{T}_{i+1}(B,y^{\otimes i}) \quad \text{and} \quad
\mathcal{G}_i(y^{\otimes i})= \mathcal{T}_i(Ny, y^{\otimes i-1}).
\end{equation}
Moreover, by definition of $\mathcal{R}_k$ (see \eqref{eqLyapounov3b}) and by Lemma \ref{lemma:symmetry}, we have
\begin{align}
\mathcal{R}_k(y^{\otimes k})
= \ & 2k(k-1) \mathcal{C}_1(y)\mathcal{G}_{k-1}(y^{\otimes k-1}) \notag \\
& \quad + \sum_{i=2}^{k-2} \begin{pmatrix} k \\ i \end{pmatrix}
\big( \mathcal{C}_i(y^{\otimes i}) + i \mathcal{G}_i(y^{\otimes i}) \big)
\big( \mathcal{C}_{k-i}(y^{\otimes k-i}) + (k-i) \mathcal{G}_{k-i}(y^{\otimes k-i}) \big).
\label{eq:rpProof2}
\end{align}
Using once again the symmetry of the multilinear forms $\mathcal{T}_k$, we obtain
\begin{equation} \label{eq:rpProof3}
k \mathcal{T}_k(A_{\Pi} y, y^{\otimes k-1})= \frac{1}{2\alpha} \mathcal{R}_k(y^{\otimes k}).
\end{equation}
We are now ready to prove \eqref{eq:resi_hjb1}.
We first have
\begin{equation} \label{eqRPExpansion1}
D\mathcal{V}_p(y) (Ay) = \sum_{k=2}^p \frac{1}{(k-1)!} \mathcal{T}_k \big( Ay,
y^{\otimes k-1} \big).
\end{equation}
Moreover, by \eqref{eq:rpProof1},
\begin{align*}
D\mathcal{V}_p(y)(Ny+B)
= & \ \sum_{i=2}^p \frac{1}{(i-1)!} \mathcal{T}_i(Ny + B, y^{\otimes i-1}) \\
= & \ \sum_{i=2}^p \frac{1}{(i-1)!} \big( \mathcal{G}_i(y^{\otimes i}) +
\mathcal{C}_{i-1}(y^{\otimes i-1}) \big) \\
= & \ \mathcal{C}_1(y) + \sum_{i=2}^{p-1} \frac{1}{i!} \big(
\mathcal{C}_i(y^{\otimes i}) + i \mathcal{G}_i(y^{\otimes i}) \big) +
\frac{1}{(p-1)!} \mathcal{G}_p(y^{\otimes p}) \\
= & \ \sum_{i=1}^{p} q_{p,i}(y).
\end{align*}
The expression $D\mathcal{V}_p(y)(Ny+B)$ is therefore the sum of monomial
functions of degree $1$,...,$p$. As a consequence, $\big(
D\mathcal{V}_p(y)(Ny+B) \big)^2$ can be expressed as a sum of monomial functions
$\tilde{q}_{p,2}$, $\tilde{q}_{p,3}$,...,$\tilde{q}_{p,2p}$ of degree
$2$,...,$2p$,
respectively:
\begin{equation} \label{eqRPExpansion0}
\big( D\mathcal{V}_p(y)(Ny+B) \big)^2 = \sum_{k=2}^{2p} \tilde{q}_{p,k}(y).
\end{equation}
We compute now these monomial functions. First,
\begin{equation*}
\tilde{q}_{p,2}(y)= q_{p,1}(y)^2= \mathcal{C}_1(y)^2= \mathcal{T}_2(B,y)^2=
\langle \Pi
B, y \rangle^2.
\end{equation*}
For $3 \leq k \leq p$, we obtain
\begin{equation} \label{eq:aux3}
\tilde{q}_{p,k}(y)= \underbrace{2 q_{p,1}(y) q_{p,k-1}(y)}_{=:(a)} +
\underbrace{\sum_{i=2}^{p-2} q_{p,i}(y) q_{p,k-i}(y)}_{=:(b)}.
\end{equation}
The terms $(a)$ and $(b)$ can be expressed explicitly as follows:
\begin{align*}
(a)= \ & \frac{2}{(k-1)!} \mathcal{C}_1(y) \big( \mathcal{C}_{k-1}(y^{\otimes
k-1}) + (k-1) \mathcal{G}_{k-1}(y^{\otimes k-1}) \big) \\
 & \qquad = \frac{2k}{k!} \mathcal{T}_k(BB^* \Pi y, y^{\otimes k-1})
+ \frac{2k(k-1)}{k!} \mathcal{C}_1(y) \mathcal{G}_{k-1}(y^{\otimes k-1}), \\
(b)= \ & \frac{1}{k!} \sum_{i=2}^{k-2} \begin{pmatrix} k \\ i \end{pmatrix}
\big( \mathcal{C}_i(y^{\otimes i}) + i \mathcal{G}_i(y^{\otimes i}) \big)
\big( \mathcal{C}_{k-i}(y^{\otimes k-i}) + (k-i) \mathcal{G}_{k-i}(y^{\otimes
k-i}) \big),
\end{align*}
and thus, using \eqref{eq:rpProof2}, relation \eqref{eq:aux3} becomes
\begin{equation} \label{eqRPExpansion2}
\tilde{q}_{p,k}(y)= \frac{2k}{k!} \mathcal{T}_k(\mathcal{B}\mathcal{B}^* \Pi y,
y^{\otimes k-1})
+ \frac{1}{k!} \mathcal{R}_k(y^{\otimes k}).
\end{equation}
For $p+1 \leq k \leq 2p$, we have
\begin{equation} \label{eqRPExpansion3}
\tilde{q}_{p,k}(y)= \sum_{i=k-p}^p q_{p,i}(y) q_{p,k-i}(y).
\end{equation}
Using \eqref{eqRPExpansion1}, \eqref{eqRPExpansion0}, \eqref{eqRPExpansion2}, and
\eqref{eqRPExpansion3}, and grouping  monomial functions of same degree,
we obtain
\begin{align*}
& - D\mathcal{V}_p(y)(Ay) -
\frac{1}{2} \| y \|_Y^2 + \frac{1}{2\alpha }
\left( D\mathcal{V}_p(y) (Ny + B)
\right)^2 \\
& \qquad =  -\frac{1}{2} \big[ 2 \mathcal{T}_2(A_\Pi y, y) + \| y \|_Y^2 + \frac{1}{\alpha}
\mathcal{T}_2(B,y)^2 \big] \\
& \qquad \qquad + \sum_{k=3}^p \frac{1}{k!} \big[ k \mathcal{T}_k(A y, y^{\otimes k-1})
- \frac{1}{2\alpha} \mathcal{R}_k(y^{\otimes k}) \big]
+ \frac{1}{2\alpha} \sum_{k=p+1}^{2p} \tilde{q}_{p,k}(y)
= r_p(y).
\end{align*}
The terms in brackets in the above expression
are equal to zero by \eqref{eq:rpProof3}.
This proves \eqref{eq:resi_hjb1}.

Let us prove \eqref{eq:resi_hjb2}. From \eqref{eq:poly_val_func} and Theorem
\ref{thm:mult_lin_form},
we obtain that for all $p \geq 2$, there exists a constant $\tilde{C}>0$ such
that for all $i=1,...,p$, $|q_{p,j}(y)| \leq \tilde{C} \| y \|_{V} \,\| y
\|_{Y}^{j-1}$. We deduce that for all $i=p,...,2p$ and all $j= i-p,...,p$,
\begin{equation*}
|q_{p,j}(y)q_{p,i-j}(y)| \leq \tilde{C}^2 \, \| y \|_{V}^2 \, \| y \|_{Y}^{i-2}.
\end{equation*}
Estimate \eqref{eq:resi_hjb2} follows then from the definition of $r_p$.
\end{proof}

\begin{lemma} \label{lemma:estimateRp}
Let $p \geq 2$ and let $\delta_0 > 0$ be the constant given by Theorem
\ref{thm:well_posed_feedback}. Then, there exists a constant $C>0$ such that for
all
$y_0 \in B_Y(\delta_0)$,
\begin{equation*}
\int_0^\infty r_p \big( \bar{y}(t) \big) \dd t \leq C \| y_0 \|_{Y}^{p+1} \quad
\text{and} \quad
\int_0^\infty r_p \big( S(\mathbf{u}_p,y_0;t) \big) \dd t \leq C \| y_0
\|_{Y}^{p+1},
\end{equation*}
where $\bar{y}$ is an optimal trajectory for problem \eqref{eqProblem} with
initial value $y_0$.
\end{lemma}

\begin{proof}
By Theorem \ref{thm:well_posed_feedback}, there exists a constant $C_1$ such
that for all $y_0 \in B_Y(\delta_0)$,
\begin{equation*}
\| y_p \|_{L^2(0,\infty;V)} \leq C_1 \| y_0 \|_Y \quad \text{and} \quad
\| y_p \|_{L^\infty(0,\infty;Y)} \leq C_1 \| y_0 \|_Y,
\end{equation*}
where $y_p= S(\mathbf{u}_p,y_0)$. By Proposition \ref{prop:bound_opt_sol},
increasing if necessary the value of $C_1>0$,  for each  solution $\bar{u}$ to
problem \eqref{eqProblem}
associated to an initial value $y_0 \in B_Y(\delta_0)$    we have
\begin{equation*}
\| \bar{y} \|_{L^2(0,\infty;V)} \leq C_1 \| y_0 \|_{Y} \quad \text{and} \quad
\| \bar{y} \|_{L^\infty(0,\infty;Y)} \leq C_1 \| y_0 \|_{Y},
\end{equation*}
where $\bar{y}= S(\bar{u},y_0)$.
Let us denote by $C_2$ the constant provided by Proposition \ref{prop:resi_hjb}.
We obtain
\begin{align*}
\int_0^\infty r_p(\bar{y}(t)) \dd t
\leq \ & C_2 \sum_{i=p+1}^{2p} \| \bar{y} \|_{L^2(0,T;V)}^2  \| \bar{y}
\|_{L^\infty(0,\infty;Y)}^{i-2} \\
\leq \ & C_2 \sum_{i=p+1}^{2p} C_1^i \| y_0 \|_{Y}^i
\leq  C_2 \| y_0 \|_{Y}^{p+1} \sum_{i=p+1}^{2p} C_1^i \delta_0^{i-(p+1)},
\end{align*}
and these inequalities also hold for $y_p$. The lemma follows  with $C= C_2
\sum_{i=p+1}^{2p} C_1^i \delta_0^{i-(p+1)}$.
\end{proof}

In the following lemma, we establish that the control  $\mathbf{U}_p(y_0)=
\mathbf{u}_p(S(\mathbf{u}_p,y_0;\cdot))$  obtained from \eqref{eq:feedback_law}
is optimal with respect to $\mathcal{J}_p(\cdot,y_0)$ for small values of $\| y_0 \|_Y$, over all feasible controls
for \eqref{eqProblem}.

\begin{lemma} \label{lemma:optimalityJp}
Let $p \geq 2$ and let $\delta_0 > 0$ be given by Theorem
\ref{thm:well_posed_feedback}. Let ${u}$
be any feasible control for \eqref{eqProblem} with initial value $y_0 \in
B_Y(\delta_0)\cap V$. Then
$\mathcal{J}_p({u},y_0)$ and $\mathcal{J}_p(\mathbf{U}_p(y_0),y_0)$ are finite
and
\begin{equation*}
\mathcal{V}_p(y_0)= \mathcal{J}_p(\mathbf{U}_p(y_0),y_0) \leq
\mathcal{J}_p({u},y_0).
\end{equation*}
\end{lemma}

\begin{proof}
We start with a computation for an arbitrary feasible control associated with an
initial condition
$y_0 \in B_Y(\delta_0)\cap V$. There exists at least one such control, namely
$\mathbf{U}_p(y_0)$.
Let us set ${y}= S({u},y_0)$.
By Lemmas \ref{lemma:RegEstim} and \ref{le2.2}, we have that ${y} \in H^1(0,T;Y)$, for every $T>0$.
Together with Lemma \ref{lemma:estimateRp}, this implies that
$\mathcal{J}_p({u},y_0)$ and $\mathcal{J}_p(\mathbf{U}_p(y_0),y_0)$ are finite.
Moreover,
for all $T>0$, we have $y \in W^{1,1}(0,T;Y)$ and by Lemma
\ref{lemma:ChainRule},
the chain rule can be applied to each of the bounded multilinear forms which
appear as  summands  in $\mathcal{V}_p(y(\cdot))$.
Omitting the time variable in what follows, we obtain
\begin{align*}
\frac{\dd}{\dd t} \mathcal{V}_p(y)= D\mathcal{V}_p(y) \big( A y + (N y +B ) u
\big)
= D\mathcal{V}_p(y)(A y) + u D\mathcal{V}_p(y) (N y +B) .
\end{align*}
By Proposition \ref{prop:resi_hjb},
\begin{align*}
 \frac{\dd}{\dd t} \mathcal{V}_p(y)&=-r_p(y) -\frac{1}{2} \| y \|_Y^2 +
\frac{1}{2\alpha}
\big( D\mathcal{V}_p(y)(Ny +B) \big)^2 + u D\mathcal{V}_p (y)
(N y +B) \\
&= -\ell_p(y,u) + \frac{1}{2\alpha}
\big( D\mathcal{V}_p(y)(Ny
+B) \big)^2 + u D\mathcal{V}_p (y)
(N y +B) + \frac{\alpha}{2}
u^2,
\end{align*}
where
\begin{align*}
\ell_p(y,u):= \frac{1}{2} \| y \|_Y^2 + \frac{\alpha}{2} u^2 + r_p(y) .
\end{align*}
Hence, it follows that
\begin{align} \label{eq:verifTh1}
 \frac{\dd}{\dd t} \mathcal{V}_p(y)&= - \ell_p(y,u) +
\frac{\alpha}{2} \left( u + \frac{1}{\alpha} D\mathcal{V}_p(y) (
N y +B ) \right)^2
= -\ell_p(y,u) + \frac{\alpha}{2} \big( u- \mathbf{u}_p(y) \big)^2.
\end{align}
We deduce that for an arbitrary feasible $u$,
\begin{equation} \label{eq:verifTh2a}
\mathcal{V}_p({y}(T)) - \mathcal{V}_p(y_0) \geq - \int_0^T
\ell_p({y},{u}) \dd t.
\end{equation}
We also deduce from \eqref{eq:verifTh1} that for the specific
$u=\mathbf{U}_p(y_0)$,
\begin{equation} \label{eq:verifTh2b}
\mathcal{V}_p(y_p(T)) - \mathcal{V}_p(y_0) = - \int_0^T
\ell_p(y_p,\mathbf{U}_p(y_0)) \dd t,
\end{equation}
since for this control, the squared expression vanishes.
By Lemma \ref{lemma:RegEstim}, we have
$\lim_{T\to \infty} y(T)=0 \text{ and } \lim_{T\to \infty}y_p(T)=0 \text{ in }
Y.$
Together with the continuity of $\mathcal{V}_p$ established in Proposition
\ref{prop:continuityValueFunction}, this implies that
\begin{equation*}
\mathcal{V}_p({y}(T)) \underset{T \to \infty}{\longrightarrow 0} \quad
\text{and} \quad
\mathcal{V}_p(y_p(T)) \underset{T \to \infty}{\longrightarrow} 0.
\end{equation*}
Finally, passing to the limit in \eqref{eq:verifTh2a} and \eqref{eq:verifTh2b}, we obtain
\begin{equation*}
\mathcal{J}_p({u},y_0)
= \int_0^\infty \ell_p({y},{u})
\geq \mathcal{V}_p(y_0)
= \int_0^\infty \ell_p(y_p,\mathbf{U}_p(y_0))
= \mathcal{J}_p(\mathbf{U}_p(y_0),y_0).
\end{equation*}
The lemma is proved.
\end{proof}

We now prove that $\mathcal{V}_p$ is a Taylor expansion of $\mathcal{V}$ and analyse the quality of the feedback law $\mathbf{u}_p$ in the neighborhood of 0.

\begin{theorem} \label{thm:subOptimality}
Let $\delta_0 >0$ be given by Theorem \ref{thm:well_posed_feedback}, let $C$ be
the constant given by Lemma \ref{lemma:optimalityJp}. Then, for all
$y_0\in B_Y(\delta_0)$,
\begin{align}
& \mathcal{J}(\mathbf{U}_p(y_0),y_0) \leq \mathcal{V}(y_0) + 2C
\|y_0\|_{Y}^{p+1}, \label{eq:subOptimality1} \\
& | \mathcal{V}(y_0)-\mathcal{V}_p(y_0) | \leq C \|y_0\|_{Y}^{p+1}.
\label{eq:subOptimality2}
\end{align}
\end{theorem}

\begin{proof}
We first prove the result for $y_0 \in B_Y(\delta_0) \cap V$. The following
inequalities follow directly from Lemma \ref{lemma:estimateRp} and Lemma
\ref{lemma:optimalityJp} and from the suboptimality of $\mathbf{U}(y_0)$:
\begin{align*}
\begin{array}{ll}
|\mathcal{V}_p(y_0) - \mathcal{J}(\mathbf{U}_p(y_0),y_0)| \leq C \| y_0
\|_{Y}^{p+1}, \phantom{\Big|} \qquad &
\mathcal{V}_p(y_0) \leq \mathcal{J}_p(\bar{u},y_0), \\
|\mathcal{V}(y_0) - \mathcal{J}_p(\bar{u},y_0)| \leq C \| y_0 \|_{Y}^{p+1},
\phantom{\Big|}&
\mathcal{V}(y_0) \leq \mathcal{J}(\mathbf{U}_p(y_0),y_0),
\end{array}
\end{align*}
where $\bar{u}$ is a solution to \eqref{eqProblem} with initial value $y_0$.
Therefore,
\begin{align*}
& \mathcal{J}(\mathbf{U}_p(y_0),y_0)-2 C \| y_0 \|_{Y}^{p+1}
\leq \mathcal{V}_p(y_0) - C \| y_0 \|_{Y}^{p+1}
\leq \mathcal{J}_p(\bar{u},y_0) - C \| y_0 \|_{Y}^{p+1} \\
& \qquad \leq \mathcal{V}(y_0)
\leq \mathcal{J}(\mathbf{U}_p(y_0),y_0)
\leq \mathcal{V}_p(y_0) + C \| y_0 \|_{Y}^{p+1},
\end{align*}
which proves inequalities \eqref{eq:subOptimality1} and
\eqref{eq:subOptimality2} for $y_0 \in B_Y(\delta_0) \cap V$. By Lemma
\ref{lemma:continuousTensors1}, $\mathcal{V}_p$ is continuous, by Proposition
\ref{prop:continuityValueFunction}, $\mathcal{V}$ is continuous on
$B_Y(\delta_0)$. By Theorem \ref{thm:well_posed_feedback} and Corollary
\ref{cor:V_loc_bound}, the mappings: $y_0 \in B_Y(\delta_0) \mapsto
S(\mathbf{U}_p(y_0),y_0)$ and $y_0 \in B_Y(\delta_0) \mapsto \mathbf{U}_p(y_0)$
are
both continuous. Moreover, the following mapping is continuous:
\begin{equation*}
(u,y) \in L^2(0,\infty) \times W_\infty \mapsto \frac{1}{2} \| y
\|_{L^2(0,T;Y)}^2 + \frac{\alpha}{2} \| u \|_{L^2(0,\infty)}^2.
\end{equation*}
Therefore, by composition, the mapping $y_0 \in B_Y(\delta_0) \mapsto
\mathcal{J}(\mathbf{U}_p(y_0),y_0)$ is continuous. Finally, since $B_Y(\delta_0)
\cap V$ is dense in $B_Y(\delta_0)$, we can pass to the limit in inequalities
\eqref{eq:subOptimality1} and \eqref{eq:subOptimality2}. They are therefore
satisfied for all $y_0 \in B_Y(\delta_0)$. The theorem follows.
\end{proof}

\begin{remark}
{\em
Inequality \eqref{eq:subOptimality1} gives an estimate for the approximation
quality of the   feedback law $\mathbf{u}_p$ in the neighborhood of 0.
In general, an inequality like \eqref{eq:subOptimality2} does not imply that
$\mathcal{V}$ is $p$-times differentiable in the neighborhood of $0$.
Indeed, consider the function
\begin{equation*}
f\colon x \in \R \mapsto
\begin{cases}
\begin{array}{cl}
x^3 \sin(1/x^2) & \text{ if $x \neq 0$} \\
0 & \text{ if $x= 0$}.
\end{array}
\end{cases}
\end{equation*}
Then, for all $x \in \R$, $|f(x)| \leq |x|^3$, however, $f$ is not continuously
differentiable at 0, since for all $x \neq 0$, $f'(x)= 3x^2 \sin(1/x^2) - 2
\cos(1/x^2)$, thus $f'(x)\nrightarrow 0$ when $x \downarrow 0$.}
\end{remark}

We finally give an error estimate for the closed-loop control $\mathbf{U}_p(y_0)$ associated with $\mathbf{u}_p$, for small values of $y_0$.

\begin{theorem} \label{thm:errorEstimForUp}
Let $\delta_0$ be given by Theorem \ref{thm:well_posed_feedback}. There exist $\delta_1 \in (0,\delta_0]$ and $C>0$ such that for all $y_0 \in B_Y(\delta_1)$, there exists a solution $\bar{u}$ to problem \eqref{eqProblem} with initial value $y_0$ satisfying the following error estimates:
\begin{equation} \label{eq:errorEstimForUp}
\| \bar{y} - S(\mathbf{u}_p,y_0) \|_{W_\infty} \leq C \| y_0 \|_Y^{(p+1)/2} \quad \text{and} \quad
\| \bar{u} - \mathbf{U}_p(y_0) \|_{L^2(0,\infty)} \leq C \| y_0 \|_Y^{(p+1)/2},
\end{equation}
where $\bar{y}= S(\bar{u},y_0)$.
\end{theorem}

\begin{proof}
The value of $\delta_1$ is fixed to $\delta_0$ for the moment. We first prove the result for $y_0 \in B_Y(\delta_1) \cap V$, as in the proof of Theorem \ref{thm:subOptimality}. Let $\bar{u}$ be a solution to problem \eqref{eqProblem} with initial condition $y_0$ and let $\bar{y}= S(\bar{u},y_0)$, $u_p= \mathbf{U}_p(y_0)$, $y_p= S(\mathbf{u}_p,y_0)$. By Theorem \ref{thm:well_posed_feedback} and Proposition \ref{prop:bound_opt_sol}, there exists a constant $C$ independent of $y_0$ such that
\begin{equation} \label{eq:estimU1}
\| \bar{y} \|_{L^\infty(0,\infty;Y)} \leq C \delta_1, \quad
\| \bar{y} \|_{L^2(0,\infty;V)} \leq C \delta_1, \quad
\| y_p \|_{L^\infty(0,\infty;Y)} \leq C \delta_1, \quad
\| y_p \|_{L^2(0,\infty;V)} \leq C \delta_1.
\end{equation}
Let us emphasize the fact that in the proof, the mapping $\mathbf{u}_p(\bar{y}(\cdot)) \in L^2(0,\infty)$ plays an important role. It can be seen as an ``intermediate" control between $\bar{u}$ and $u_p$.

\emph{Step 1}: estimation of $\| \bar{u}(\cdot) - \mathbf{u}_p(\bar{y}(\cdot)) \|_{L^2(0,\infty)}$.
Since $y_0 \in V$, equality \eqref{eq:verifTh1} holds for $(u,y)= (\bar{u},\bar{y})$ and therefore, for a.e.\@ $t \geq 0$,
\begin{equation*}
\frac{\dd}{\dd t} \mathcal{V}_p(\bar{y}(t))
= - \ell_p(\bar{y}(t),\bar{u}(t)) + \frac{\alpha}{2} \big( \bar{u}(t)- \mathbf{u}_p (\bar{y}(t) )\big)^2.
\end{equation*}
Integrating on $[0,T]$ and passing to the limit when $T \rightarrow \infty$, as in the proof of Lemma \ref{lemma:optimalityJp}, we obtain that
\begin{equation*}
-\mathcal{V}_p(y_0)
= - \underbrace{\int_0^\infty \ell(\bar{y}(t),\bar{u}(t)) \dd t}_{=\mathcal{V}(y_0)}
- \int_0^\infty r_p(\bar{y}(t)) \dd t
+ \frac{\alpha}{2} \int_0^\infty \big( \bar{u}(t)- \mathbf{u}_p(\bar{y}(t)) \big)^2 \dd t
\end{equation*}
and finally that
\begin{equation} \label{eq:estimU2}
\| \bar{u}( \cdot ) - \mathbf{u}_p(\bar{y}(\cdot)) \|_{L^2(0,\infty)}^2
\leq \frac{2}{\alpha} \Big( |\mathcal{V}_p(y_0)-\mathcal{V}(y_0)| + \int_0^\infty | r_p(\bar{y}(t)) | \dd t \Big) \leq C \| y_0 \|_Y^{p+1},
\end{equation}
as a consequence of Theorem \ref{thm:subOptimality} and Lemma \ref{lemma:estimateRp}.

\emph{Step 2:} estimation of $\| \bar{y}- y_p \|_{W_\infty}$.
We use in this part of the proof ideas similar to the ones developed for the well-posedness of the closed-loop system in Theorem \ref{thm:well_posed_feedback}. We make use of the mapping $F$, defined by \eqref{eq:nonlinearity}. Remember that this mapping contains the non-linearities of the closed-loop system (see \eqref{eq:nonlinearities}).
Let us set
\begin{equation*}
f(t)=  \big( N \bar{y}(t) + B \big) \big( \bar{u}(t)- \mathbf{u}_p(\bar{y}(t) \big) \in V^*, \quad
\text{for a.e.\@ $t \geq 0$}.
\end{equation*}
Omitting the time variable, we have
\begin{equation*}
\frac{\dd}{\dd t} \bar{y}= A \bar{y} + (N \bar{y} + B)(\bar{u}- \mathbf{u}_p(\bar{y}))
+ (N\bar{y} + B) \mathbf{u}_p(\bar{y}) = A_\Pi \bar{y} + F(\bar{y}) + f.
\end{equation*}
We also have
\begin{equation*}
\frac{\dd}{\dd t} y_p= A_\Pi y_p + F(y_p).
\end{equation*}
Setting $z= \bar{y} - y_p$, we obtain
\begin{equation*}
\frac{\dd}{\dd t} z= A_\Pi z + F(\bar{y}) - F(y_p) + f, \quad z(0)= 0.
\end{equation*}
We compute now estimates of $\| F(\bar{y})-F(y_p) \|_{L^2(0,\infty;V^*)}$ and $\| f \|_{L^2(0,\infty;V^*)}$, in order to obtain an estimate of $\| z \|_{W_\infty}$ with Proposition \ref{prop:reg_nonh}. By definition of $f$, we have
\begin{align}
\| f \|_{L^2(0,\infty;V^*)}^2
\leq \ & 2\big( \| N \|_{\mathcal{L}(Y,V^*)}^2 \| \bar{y} \|_{L^\infty(0,\infty;Y)}^2 + \| B \|_{V^*}^2 \big) \| \bar{u}(\cdot) - \mathbf{u}_p(\bar{y}(\cdot)) \|_{L^2(0,\infty)}^2 \notag \\
\leq \ & C \| y_0 \|_{Y}^{p+1}, \label{eq:estimU3}
\end{align}
where the last inequality follows from the estimates \eqref{eq:estimU1} and \eqref{eq:estimU2}.
Since $\| \bar{y} \|_{L^\infty(0,\infty;Y)} \leq C \delta_1$ and $\| \bar{y}_p \|_{L^\infty(0,\infty;Y)} \leq C \delta_1$, we obtain with Lemma \ref{lem:LipschitzNonLin} that
\begin{align}
\| F(\bar{y})-F(y_p) \|_{L^2(0,\infty;V^*)}
\leq \ & C_1 \big( C \delta_1  + \| \bar{y} \|_{L^2(0,\infty;V)} + \| y_p \|_{L^2(0,\infty;V)} \big) \| \bar{y}-y_p \|_{W_\infty} \notag \\
\leq \ & 3 C_1 C \delta_1 \| z \|_{W_\infty}. \label{eq:estimU4}
\end{align}
We can now reduce the value of $\delta_1$ to
\begin{equation*}
\delta_1= \min \Big( \delta_0, \frac{1}{6C_1 C_2 C} \Big).
\end{equation*}
By \eqref{eq:estimU3}, \eqref{eq:estimU4}, and Proposition \ref{prop:reg_nonh},
\begin{align*}
\| z \|_{W_\infty}
\leq \ & C_2 \big( \| F(\bar{y})- F(y_p) \|_{L^2(0,\infty;V^*)} + \| f \|_{L^2(0,\infty;V^*)} \big) \\
\leq \ & 3 C_1 C_2 C \delta_1 \| z \|_{W_\infty} + C_2 C \| f \|_{L^2(0,\infty;V^*)} \\
\leq \ & \frac{1}{2} \| z \|_{W_\infty} + C \| y_0 \|_{Y}^{(p+1)/2}.
\end{align*}
It follows that
\begin{equation*}
\| z \|_{W_\infty} \leq C \| y_0 \|_Y^{(p+1)/2}.
\end{equation*}
The first estimate in \eqref{eq:errorEstimForUp} is now proved.

\emph{Step 3:} estimation of $\| \bar{u} - u_p \|_{L^2(0,\infty)}$.
Observing that $u_p(\cdot)= \mathbf{u}_p(y_p(\cdot))$, we obtain that
\begin{align}
\| \bar{u} - u_p \|_{L^2(0,\infty)}
\leq \ & \| \bar{u}(\cdot) - \mathbf{u}_p(\bar{y}(\cdot)) \|_{L^2(0,\infty)} + \| \mathbf{u}_p(\bar{y}(\cdot)) - u_p(\cdot) \|_{L^2(0,\infty)} \notag \\
\leq \ & C \| y_0 \|^{(p+1)/2} + \| \mathbf{u}_p(\bar{y}(\cdot)) - \mathbf{u}_p(y_p(\cdot)) \|_{L^2(0,\infty)}. \label{eq:estimU5}
\end{align}
We obtain an estimate of the last term of the r.h.s.\@ by proving a Lipschitz property for the mapping $y \in W_\infty \mapsto \mathbf{u}_p(y(\cdot)) \in L^2(0,\infty)$.
With similar estimates to the ones used in the proof of Lemma \ref{lem:Lipschitz1}, one can easily show that for $y_1$ and $y_2 \in B_Y(C \delta_1)$, for all $k=2,...,p$,
\begin{equation*}
|\mathcal{T}_k(Ny_2 + B, y_2^{\otimes k-1}) - \mathcal{T}_k(N y_1 + B, y_1^{\otimes k-1 })|
\leq C \big( \| y_2 - y_1 \|_{V} + \| y_1 \|_V \| y_2 - y_1 \|_Y \big).
\end{equation*}
By \eqref{eq:estimU1}, $\| \bar{y} \|_{L^\infty(0,\infty;Y)} \leq C \delta_1$ and $\| y_p \|_{L^\infty(0,\infty;Y)} \leq C \delta_1$. Therefore,
\begin{align*}
\| \mathbf{u}_p(\bar{y}(\cdot)) - \mathbf{u}_p(y_p(\cdot)) \|_{L^2(0,\infty)}^2
\leq \ & C \big( \| \bar{y}- y_p \|_{L^2(0,\infty;V)}^2 + \| y_p \|_{L^2(0,\infty;V)}^2 \| \bar{y} - y_p \|_{L^\infty(0,\infty;Y)}^2 \big) \\
\leq \ & C \| y_p - \bar{y} \|_{W_\infty}^2 \leq C \| y_0 \|_Y^{p+1}.
\end{align*}
Combining this estimate with \eqref{eq:estimU5}, we obtain the second inequality of \eqref{eq:errorEstimForUp}.

\emph{Step 4:} general case.
Let $y_0 \in B_Y(\delta_1)$. Take a sequence $(y_0^k)_{k \in \mathbb{N}}$ in $B_Y(\delta_1) \cap V$ converging to $y_0$. As we proved in the first three steps of this proof, for all $k \in \mathbb{N}$, there exists a solution $\bar{u}^k$ to problem \eqref{eqProblem} with initial condition $y_0^k$ such that
\begin{equation} \label{eq:estimU6}
\| \bar{y}^k - S(\mathbf{u}_p,y_0^k) \|_{W_\infty} \leq C \| y_0^k \|_Y^{(p+1)/2} \quad \text{and} \quad \| \bar{u}^k - \mathbf{U}_p(y_0^k) \|_{L^2(0,\infty)}
\leq C \| y_0^k \|_Y^{(p+1)/2},
\end{equation}
where $\bar{y}^k= S(\bar{u}^k,y_0^k)$.
Using arguments similar to the ones used in the proof of Proposition \ref{prop:kk1}, we obtain that there exists an accumulation point $(\bar{u},\bar{y})$ to the sequence $(\bar{u}^k,\bar{y}^k)$ for the weak topology of $L^2(0,\infty) \times W_\infty$ which is such that $\bar{u}$ is a solution to problem \eqref{eqProblem} with initial condition $y_0$ and such that $\bar{y}= S(\bar{u},y_0)$. By Corollary \ref{cor:V_loc_bound}, the mapping $\mathbf{U}_p$ is continuous. Therefore, we can pass to the limit in \eqref{eq:estimU6} and finally obtain the estimates
\begin{equation*}
\| \bar{y} - S(\mathbf{u}_p,y_0) \|_{W_\infty} \leq C \| y_0 \|_Y^{(p+1)/2} \quad \text{and} \quad
\| \bar{u} - \mathbf{U}_p(y_0) \|_{L^2(0,\infty)} \leq C \| y_0 \|_Y^{(p+1)/2},
\end{equation*}
which concludes the proof.
\end{proof}

\begin{remark}
The constants $\delta_0$, $\delta_1$, and $C$, which are provided by Theorem \ref{thm:subOptimality} and Theorem \ref{thm:errorEstimForUp}, depend on $p$.
\end{remark}

\section{Stabilization of a Fokker-Planck equation}\label{sec:FP}

In this section, we show that assumptions \eqref{ass:A1}-\eqref{ass:A4} are satisfied for a
concrete infinite-dimensional bilinear  optimal control problem.
Following the setup discussed in \cite{BreKP16}, we focus on the controlled
Fokker-Planck equation
\begin{equation}\label{eq:FP_setup}
  \begin{aligned}
  \frac{\partial \rho}{\partial t} &= \tilde \nu \Delta \rho +
\nabla \cdot (\rho \nabla G)+ u \nabla \cdot (\rho \nabla \alpha)  && \text{in
} \Omega \times (0,\infty), \\
0 &= (\tilde \nu \nabla \rho + \rho \nabla G) \cdot \vec{n} && \text{on }
\Gamma \times
(0,\infty), \\
\rho(x,0) &= \rho_0(x) && \text{in } \Gamma,
  \end{aligned}
\end{equation}
where $\tilde \nu>0,\; \Omega \subset \mathbb R^n$ denotes a bounded domain with
smooth
boundary $\Gamma = \partial \Omega$, and $\rho_0$  denotes an
initial probability distribution with $\int_{\Omega} \rho_0(x)\mathrm{d}x =1.$
To apply the results from \cite{BreKP16}, we assume that $\alpha$ and $G \in W^{1,\infty} \cap
W^{2,\max(2,n)}(\Omega)$, and that the control shape
function fulfills $\nabla \alpha \cdot \vec{n} = 0$ on $\Gamma.$ We introduce
$\rho_\infty= \frac{e^{-\Phi}}{\int_\Omega e^{-\Phi}\dd x},$ where $\Phi=\log
\tilde \nu + \frac{W}{\tilde \nu}$, and observe that $\rho_\infty$ is an
eigenstate associated with the eigenvalue $0$.
While the system is known to converge to this stationary distribution, this
can happen inadequately slowly and a control mechanism becomes relevant.
Considering \eqref{eq:FP_setup} as an abstract bilinear control, we arrive at
\begin{align*}
  \dot{\rho}(t)&= A \rho(t) + N \rho(t) u(t),
\quad \rho(0) =\rho_0,
\end{align*}
where the operators $A$ and $N$ are given by
\begin{equation*}
\begin{aligned}
 A\colon \mD(A)&\subset L^2(\Omega) \to
L^2(\Omega),\\
\mD(A)&= \left\{\rho \in H^2(\Omega) \left| (\tilde \nu \nabla \rho +
\rho \nabla G) \cdot \vec{n}  =0 \text{ on } \Gamma \right. \right\}, \\
A\rho& = \tilde \nu \Delta \rho + \nabla \cdot (\rho \nabla G), \\[1ex]
N\colon H^1(\Omega)& \to L^2(\Omega),\ \  N\rho =
\nabla \cdot (\rho \nabla \alpha).
\end{aligned}
\end{equation*}
In order to consider \eqref{eq:FP_setup} as a stabilization problem of the form
\eqref{eqProblem}, we introduce a state variable $y:=\rho-\rho_\infty$ as the
deviation to the stationary distribution. As discussed in \cite{BreKP16}, this
yields a system of the form
\begin{align*}
  \dot{y}(t)&= A y(t) + N \rho(t) u(t) + Bu(t),
\quad y(0) =\rho_0-\rho_\infty,
\end{align*}
where
\begin{align*}
  B\colon \mathbb R &\to L^2(\Omega), \ \ B c=
 cN\rho_\infty.
\end{align*}
Since $\int_\Omega B\,\dd x = \int _\Omega N{\rho_\infty}\, \dd x =0, $
the control does not influence the one-dimensional subspace associated with
$\rho_\infty.$ Therefore, a splitting of the state space in the form
\begin{align*}
  Y = L^2(\Omega)= \mathrm{im}(P) \oplus
\mathrm{im}(I-P)=: Y_P +
Y_Q
\end{align*}
by means of the projection $P$ defined by
\begin{equation*}
  \begin{aligned}
    &P\colon L^2(\Omega)\to L^2(\Omega), \quad Py = y -
 \int_{\Omega} y \;\mathrm{d}x \;  \rho_\infty, \\
&\mathrm{im}(P)  =\left\{ v \in L^2(\Omega)\colon
\int_\Omega v\; \dd x = 0 \right\}, \quad
\mathrm{ker}(P) = \mathrm{span}\left\{\rho_\infty \right\},
  \end{aligned}
\end{equation*}
was introduced in \cite{BreKP16}.
We thus focus on
\begin{equation} \label{eq:FP_dec}
  \begin{aligned}
   \dot{y}_P &= \widehat{A} y_P +
\widehat{N} y_{P} u + \widehat{B} u,  \ \
y_P(0) = P \rho _0 ,
  \end{aligned}
\end{equation}
where
\begin{align*}
  \widehat{A}&=PA
I_P \ \text{ with }\mD(\widehat{A})=
\mD(A)\cap
Y_P , \\
\widehat{N} &= PNI_P \ \text{ with }
\mD(\widehat{N}) =H^1(\Omega) \cap Y_P
 , \\
\widehat{B} &= PB,
\end{align*}
and $I_P\colon Y_P \to Y$ denotes the
injection of $Y_P$ into $Y.$ With system
\eqref{eq:FP_dec}, we associate the cost functional
\begin{equation}\label{eq:FP_cost_func1}
  \mathcal{J}(u,\rho_0)=\frac{1}{2}\int_0^\infty \|y_P(t) \|^2_{L^2(\Omega)} \;
\dd t
+ \frac{\alpha}{2}
\int_0^\infty u(t)^2 \; \dd t.
\end{equation}
Let us verify that the  assumptions \eqref{ass:A1}-\eqref{ass:A4} are
satisfied with $Y=Y_P$ and
$V=H^1(\Omega)\cap Y_P$, endowed with the inner products
from $L^2(\Omega)$ and $H^1(\Omega)$ respectively and for the bilinear system \eqref{eq:FP_dec} with operators $\widehat{A}$, $\widehat{N}$, and $\widehat{B}$.
Concerning \eqref{ass:A1}, we
have for every $v \in V$ that
\begin{align*}
a(v,v) = \ & \langle \tilde \nu \nabla v +v \nabla G, \nabla v \rangle_{L^2(\Omega)} \\
 \geq \ & \tilde \nu \|\nabla v\|^2_{L^2(\Omega)} -|(v \nabla G,\nabla
v)_{L^2(\Omega)}| \ge \frac{\tilde \nu}{2}\|\nabla v\|^2_{L^2(\Omega)} -
\frac{1}{2\tilde \nu} \|\nabla G\|^2_{L^\infty(\Omega)}\|v\|^2_{L^2(\Omega)}.
\end{align*}
Thus \eqref{ass:A1} holds with $\nu=\frac{\tilde \nu}{2}$ and $\lambda =
\frac{1}{2\tilde \nu} \|\nabla G\|^2_{L^\infty(\Omega)}$
Using that $P^*y =y -\int_{\Omega} \rho_\infty y\, \dd x \,
\mathbbm{1}$, we further obtain that
\begin{equation*}
\widehat{N}^*\phi = I_P^* N^* P^*\phi =
I_P^* N^* \phi = - I_P^*(\nabla \phi
\nabla\alpha),
\end{equation*}
since $\nabla \alpha \cdot \vec{n} = 0$ and
$I_P^* \psi = \psi - \frac{1}{\Omega} \int_\Omega \psi \,dx \;
\mathbbm{1} $. It is now clear that \eqref{ass:A2} holds.
Assumption \eqref{ass:A3} is satisfied with $V= H^1(\Omega) \cap
Y_P$, see e.g.\@ \cite[Part II, Chapter 1, Section 6]{Benetal07}. Finally, the
exponential stability of the uncontrolled system \eqref{eq:FP_dec} (i.e.\@ with $u= 0$) implies assumption \eqref{ass:A4} with $F=0$, see \cite[Section 4]{BreKP16}.

\section{Conclusions}

Techniques for the computation of a Taylor expansion of the value function associated with an optimal control problem have been extended to the case of an infinite-dimensional bilinear system.
Explicit formulas have been derived for the right-hand side of the generalized Lyapunov equations arising for the terms of order three and more.
Non-linear feedback laws have been derived from the Taylor expansions. Their efficiency has been proved theoretically with new error estimates.
It is planned to investigate the use of the resulting generalized Lyapunov equations together with model reduction techniques in an independent study. Generalizations of our results in several directions are possible and can be of interest.
These include the case of vector-valued controls and more general dynamical systems.

\section*{Acknowledgements}

This work was partly supported by the ERC advanced grant 668998 (OCLOC) under the EU's H2020 research program.

\appendix
\section{Proofs}

In this Appendix, we provide the proofs for several results which were used in the main part of the manuscript.

\begin{proof}[Proof of Lemma \ref{lemma:RegEstim}]
The existence can be proved by standard Galerkin arguments and the
a-priori estimates below. To verify these estimates
and  to alleviate the notation, we often omit the time variable $t$.
We first prove estimates \eqref{eq:KK7} and \eqref{eq:RegEstim1}. Multiplying
the state equation by
$y$ and using \eqref{ass:A1},
 we obtain
\begin{align}
\frac{1}{2} \frac{\dd}{\dd t} \| y \|_Y^2
= \ & \left\langle \frac{\dd y}{\dd t},y \right\rangle_{V^*,V}
= \langle Ay, y \rangle_{V^*,V} + \langle Ny, y \rangle_Y u + \langle B,y
\rangle_{V^*,V} \; u
\notag \\
\leq \ & \big( \lambda \| y \|_Y^2 -  \nu \| y \|_V^2 \big) + \big( \| N
\|_{\mathcal{L}(V,Y)} \; \| y \|_V \;\| y \|_Y \; |u| \big) + \big( \|
B \|_{V^*} \; \| y \|_V \; |u| \big). \label{eq:RegEstim3}
\end{align}
By Young's inequality,
\begin{equation} \label{eq:RegEstim4}
 \begin{aligned}
  \| N \|_{\mathcal{L}(V,Y)} \; \| y \|_V \; \| y \|_Y \; |u|
 &\leq \frac{\nu}{4} \| y \|_V^2 + C \| y \|_Y^2 \; |u|^2, \\
\| B \|_{V^*} \; \| y \|_V \;  |u| &
\leq \frac{\nu}{4} \| y \|_V^2 + C  \; |u|^2.
 \end{aligned}
\end{equation}
Therefore, combining \eqref{eq:RegEstim3} and \eqref{eq:RegEstim4},
\begin{equation} \label{eq:RegEstimKey}
\frac{\dd}{\dd t} \| y \|_Y^2 + \nu \| y \|_V^2
\leq C \big( \|y \|_Y^2 + |u|^2 + \| y \|_Y^2 \  |u|^2  \big).
\end{equation}
We integrate \eqref{eq:RegEstimKey} (without the term $\nu \| y \|_V^2$) and
apply Gronwall's inequality: for all $t \in [0,T]$,
\begin{align*}
\| y(t) \|_Y^2 \leq \ & \Big( \| y_0 \|_Y^2 + C \int_0^t |u|^2 \Big)
e^{  C \int_0^t 1 + |u|^2  }\\
\leq \ & \big( \|y_0\|_Y^2 + C\| u \|_{L^2(0,T)} \big) e^{   C(T + \| u
\|_{L^2(0,T)}) }.
\end{align*}
Estimate \eqref{eq:RegEstim1} is proved. Using \eqref{eq:RegEstimKey}  once
again, together with
\eqref{eq:RegEstim1} and the state equation \eqref{eq2.1}, estimate
\eqref{eq:KK7} follows.
Let us prove \eqref{eq:RegEstim2}. Let us set $\delta y=
S(u,\tilde{y}_0)-S(u,y_0)$. We have
\begin{equation*}
\frac{\dd }{\dd t}\delta y(t)= A \delta y(t) + N \delta y(t) u(t),
\end{equation*}
therefore, using the same techniques as for the derivation of
\eqref{eq:RegEstimKey}, we obtain
\begin{equation*}
\frac{\dd}{\dd t} \| \delta y \|_Y^2
\leq C \big( \|\delta y \|_Y^2 + \| \delta y \|_Y^2 \ |u|^2  \big),
\end{equation*}
and finally, by Gronwall's inequality,
\begin{equation*}
\| \delta y(t) \|_Y^2  \leq \| \delta y(0) \|_Y^2 \
e^{ C \int_0^t 1 + |u|^2 }
\leq \| \tilde{y}_0-y_0 \|_Y^2 \ e^{ C(T + \| u \|_{L^2(0,T)}) }.
\end{equation*}
Estimate \eqref{eq:RegEstim2} is proved.

We now assume that $y \in L^2(0,\infty;Y)$.
We integrate estimate \eqref{eq:RegEstimKey} and obtain
\begin{equation*}
\| y(t) \|_Y^2 \leq \| y_0 \|_Y^2 + C \big( \| y \|_{L^2(0,\infty;Y)}^2 + \| u
\|_{L^2(0,\infty)}^2 \big)
+ C \int_0^t \| y(s) \|_Y^2 \  |u(s)|^2 \dd s.
\end{equation*}
Estimate \eqref{eq:RegEstimBis1} follows with Gronwall's inequality.
From \eqref{eq:RegEstimKey}, we also obtain
\begin{equation*}
\nu \| y \|_V^2 \leq C \big( \|y \|_Y^2 + |u|^2 + \| y \|_Y^2 \  |u|^2
\big).
\end{equation*}
Estimate \eqref{eq:RegEstimBis2} follows directly by integration.
Finally, for a.e.\@ $t \geq 0$,
\begin{equation*}
\left\| \frac{\dd y }{\dd t} \right\|_{V^*}^2
\leq 3 \Big( \| A \|_{\mathcal{L}(V,V^*)}^2 \| y \|_V^2 + \| N
\|_{\mathcal{L}(Y,V^*)}^2 \| y \|_Y^2 | u |^2 + \| B \|_Y^2 |u|^2 \Big).
\end{equation*}
Estimate \eqref{eq:RegEstimBis3} follows directly by integration.

To verify the asymptotic behavior, we use the fact that $y \in L^2(0,T,Y)$ and $y\in
C([0,T],Y)$ imply the existence of a monotonically increasing sequence of
numbers $(t_k)_{k=1}^\infty$ such that $\|y(t_k)\|_Y \to 0, t_k \to \infty$
as $k\to \infty$. Since $y \in W(0,\infty)$ for any $T>0$, we have that
\begin{align*}
 \frac{\dd }{\dd t} \langle y,y \rangle  = 2 \left\langle \frac{\dd }{\dd t}
y, y \right\rangle _{V^*,V} \text{ for a.e. } t>0,
\end{align*}
see \cite[Proposition 1.2, Chapter 3]{Sho97}. Given any $T>0$ and choosing $t_k
> T$, we estimate
\begin{align*}
  \| y(T) \| _Y^2 &= \| y(t_k) \|_Y^2 - 2\int_T^{t_k} \left \langle \frac{\dd
}{\dd t}y(t), y(t) \right\rangle_{V^*,V} \dd t  \\
&\le \| y(t_k) \|_Y^2 + 2 \left\| \frac{\dd }{\dd t} y \right \|
_{L^2(T,\infty; V^*)} \| y\| _{L^2(T,\infty,V)} \longrightarrow 0
\end{align*}
for $t_k \to \infty, T\to \infty.$
\end{proof}

\begin{proof}[Proof of Proposition \ref{prop:kk1}]
Since there exists a feasible control and since
$\mathcal{J}$ is bounded from below, $\mathcal{V}(y_0)$ is finite and
there exists a minimizing sequence $(u_n)_{n \in \mathbb{N}}$ in $L^2(0,\infty)$
with associated states $y_n:= S(u_n,y_0)$. By definition of $\mathcal{J}$, the
sequences $(u_n)_{n \in \mathbb{N}}$ and $(y_n)_{n \in \mathbb{N}}$ are bounded
in $L^2(0,\infty)$ and $L^2(0,\infty;Y)$, respectively. We deduce from estimates
\eqref{eq:RegEstimBis1}, \eqref{eq:RegEstimBis2}, and \eqref{eq:RegEstimBis3},
that the sequence $(y_n)_{n \in \mathbb{N}}$ is bounded in $W(0,\infty)$.
Extracting if necessary a subsequence, there  exists $(\bar{u},\bar{y}) \in
L^2(0,\infty) \times W(0,\infty)$ such that $(u_n,y_n) \rightharpoonup
(\bar{u},\bar{y})$ in $L^2(0,\infty) \times W(0,\infty)$.

We prove now that $\bar{y}= S(\bar{u},y_0)$. Let $T>0$, let $v \in W(0,T) \cap
L^\infty(0,T;V)$ be arbitrary. For all $n \in \mathbb{N}$, we have
\begin{equation} \label{eq:ExistenceSol1}
\int_0^T \Big\langle \frac{\dd}{\dd t} \, y_n (t), v(t) \Big\rangle_{V^*,V} \dd
t
= \int_0^T \big\langle Ay_n(t) + Ny_n(t) u_n(t) + Bu_n(t), v(t)
\big\rangle_{V^*,V} \dd t.
\end{equation}
Since $\frac{\dd }{\dd t} y_n \rightharpoonup \frac{\dd}{\dd t} \bar{y}$ in
$L^2(0,T;V^*)$, we can pass to the limit in the l.h.s.\@ of the above equality.
Moreover, since $Ay_n \rightharpoonup A \bar{y}$ in $L^2(0,T;V^*)$,
\begin{equation*}
\int_0^T \langle Ay_n(t), v(t) \rangle_{V^*,V} \dd t
\underset{n \to \infty}{\longrightarrow}
\int_0^T \langle A\bar{y}(t),v(t) \rangle_{V^*,V} \dd t.
\end{equation*}
We also have
\begin{align}
& \Big| \int_0^T \langle N y_n(t), v(t)  \rangle_{V^*,V} \, u_n(t) \dd t -
\int_0^T
\langle N\bar{y}(t), v(t) \rangle \, u(t) \dd t \Big| \notag \\
& \qquad =  \Big| \int_0^T \langle y_n(t), N^* v(t) \rangle_Y \, u_n(t) \dd t -
\int_0^T
\langle \bar{y}(t), N^*v(t) \rangle_Y \, u(t) \dd t \Big| \notag \\
& \qquad \leq \int_0^T \big| \langle y_n(t)-\bar{y}(t),N^*v(t) \rangle_Y \big| \, |u_n|
\dd t
+ \Big| \int_0^T \langle \bar{y}(t), N^*v(t) \rangle_Y \big( u_n(t)- \bar{u}(t) \big) \dd t
\Big|. \label{eq:ExistenceSol2}
\end{align}
The first integral in the r.h.s.\@ of \eqref{eq:ExistenceSol2} is bounded by
\begin{equation*}
\| y_n- \bar{y} \|_{L^2(0,T;Y)} \, \| N^* v  \|_{L^\infty(0,T;Y)} \Big( \sup_{k \in
\mathbb{N}} \, \| u_k \|_{L^2(0,T)} \Big)
\end{equation*}
and therefore converges to 0, since  $\| y_n- \bar{y} \|_{L^2(0,T;Y)} \underset{n \to
\infty}{\longrightarrow} 0$ by the Aubin-Lions lemma. Since $\bar{y} \in C(0,T;Y)$ and $N^*v \in
L^2(0,T;Y)$, it holds that: $\langle \bar{y}(\cdot),N^* v(\cdot) \rangle \in L^2(0,T)$
and therefore,
the second integral in \eqref{eq:ExistenceSol2} converges to 0.
We can pass to the limit in \eqref{eq:ExistenceSol1}. We thus obtain:
\begin{equation*}
\int_0^T \Big\langle \frac{\dd}{\dd t} \, \bar{y} (t), v(t) \Big\rangle_{V^*,V} \dd t
= \int_0^T \big\langle A \bar{y} (t) + N \bar{y} (t) u(t) + B \bar{u}(t), v(t) \big\rangle_{V^*,V} \dd
t.
\end{equation*}
Since $W(0,T) \cap L^\infty(0,T;V)$ is dense in $W(0,T)$, we obtain that $\bar{y} = S(\bar{u},y_0)$.

Finally, since the following mapping is convex:
\begin{equation*}
(u,y) \in L^2(0,\infty) \times W(0,\infty) \mapsto \frac{1}{2} \int_0^\infty \|
y(t) \|_Y^2 \dd t + \frac{\alpha}{2} \int_0^\infty |u(t)|^2 \dd t,
\end{equation*}
it is also weakly lower semi-continuous and therefore,
\begin{equation*}
\mathcal{J}(\bar{u},y_0) \leq \underset{n \to \infty}{\liminf} \ \mathcal{J}(u_n,y_0),
\end{equation*}
which proves the optimality of $\bar{u}$.
\end{proof}

\begin{proof}[Proof of Lemma \ref{lemma:continuousTensors1}]
One can easily check that if $\mathcal{T}$ is not bounded, then it is not
continuous at $0$.
Assume now that $\mathcal{T}$ is bounded. Let $M >0$, let $y=(y_1,...,y_k) \in
Y^k$ and $v=(v_1,...,v_k) \in Y^k$ be such that $\| y \|_{Y^k} \leq M$ and $\| v
\|_{Y^k} \leq M$.
Then, by \eqref{eq:OperatorNormTensor2},
\begin{align}
\big| \mathcal{T}(v_1,...,v_k)-\mathcal{T}(y_1,...,y_k) \big|
= \ & \Big| \big[ \mathcal{T}(v_1,...,v_k) - \mathcal{T}(y_1, v_2,...,v_k) \big]
\notag \\
& \qquad + \big[ \mathcal{T}(y_1,v_2,...,v_k)- \mathcal{T}(y_1,y_2,v_3,...,v_k)
\big] \notag \\
& \qquad + ... + \big[ \mathcal{T}(y_1,...,y_{k-1},v_k)-
\mathcal{T}(y_1,...,y_k) \big] \Big| \notag \\
= \ & \big| \mathcal{T}(v_1-y_1,v_2,...,v_k) +
\mathcal{T}(y_1,v_2-y_2,v_3,...,v_k) \notag \\
& \qquad + ... + \mathcal{T}(y_1,...,y_{k-1},v_k-y_k) \big| \notag \\
\leq \ & k M^{k-1} \, \| \mathcal{T} \| \, \| y-v \|_{Y^k}.
 \label{eq:LipschitzContinuity}
\end{align}
The lemma is proved.
\end{proof}

\begin{proof}[Proof of Lemma \ref{lemma:leibnitz}]
We prove the lemma by induction. The case $k=1$ is trivially satisfied, since $S_{0,1}$ and $S_{1,0}$ both consist of the unique permutation of the set $\{ 1 \}$.

Let $k \geq 1$, let us assume that formula \eqref{eq:leibnitz0} holds. Before proving \eqref{eq:leibnitz0} for $k+1$, we make an important observation on the structure of $S_{i,k+1-i}$, for $i=1,...,k$. For any $\sigma \in S_{i,k+1-i}$, either $\sigma(i)= k+1$ or $\sigma(k+1)= k+1$. More precisely, we can describe $S_{i,k+1-i}$ as follows:
\begin{align}
& S_{i,k+1-i}= \big\{ \sigma \in S_{k+1} : \exists \rho \in S_{i,k-i}, \big( \sigma(1),...,\sigma(k+1) \big) = \big( \rho(1),...,\rho(k), k+1 \big) \big\} \notag \\
& \quad  \cup \big\{ \sigma \in S_{k+1} : \exists \rho \in S_{i-1,k+1-i}, \big( \sigma(1),...,\sigma(i+j) \big) = \big( \rho(1),...,\rho(i-1),k+1, \rho(i),...,\rho(k) \big) \big\} . \label{eq:leibnitz1}
\end{align}
Let us assume that $f$ and $g$ are $(k+1)$-times differentiable. Let $(z_1,...,z_{k+1}) \in Y^{k+1}$, using the induction assumption and the fact that $|S_{i,k-i}|= \binom{k}{i}$, we obtain
\begin{align}
& D^{k+1} [f(y)g(y)] (z_1,...,z_{k+1}) \notag \\
=  & D\Big[ \sum_{i= 0}^k \sum_{\rho \in S_{i,k-i}} D^i f(y)(z_{\rho(1)},...,z_{\rho(i)}) D^{k-i} g(y) (z_{\rho(i+1)},..., z_{\rho(k)}) \Big] z_{k+1} \notag \\
= & \underbrace{\sum_{i= 0}^k \sum_{\rho \in S_{i,k-i}} D^{i+1} f(y)(z_{\rho(1)},...,z_{\rho(i)},z_{k+1})
D^{k-i} g(y) (z_{\rho(i+1)},..., z_{\rho(k)}) }_{=:(a)} \notag \\
& \qquad + \underbrace{\sum_{i= 0}^k \sum_{\rho \in S_{i,k-i}} D^{i} f(y)(z_{\rho(1)},...,z_{\rho(i)})
D^{k-i+1} g(y) (z_{\rho(i+1)},..., z_{\rho(k)}, z_{k+1})}_{=:(b)}. \label{eq:leibnitz2}
\end{align}
In the sum involved in term $(a)$, we isolate the value $i= k$. Note that $S_{k,0}$ only contains one permutation, the identity on $\{ 1,...,k \}$. We also perform a change of index for the remaining values of $i$. We finally obtain for term $(a)$ the following expression:
\begin{align}
(a)= & \sum_{i=1}^k \sum_{\rho \in S_{i-1,k+1-i}} D^{i} f(y)(z_{\rho(1)},...,z_{\rho(i-1)},z_{k+1}) D^{k+1-i} g(y) (z_{\rho(i)},...,z_{\rho(k)}) \notag \\
& \quad + D^{k+1}f(y)(z_1,...,z_{k+1}) g(y). \label{eq:leibnitz3}
\end{align}
Observe that the last term of the r.h.s.\@ can be written as follows:
\begin{equation}
D^{k+1} f(y)(z_1,...,z_{k+1}) g(y)
= \sum_{\rho \in S_{k+1,0}} D^{k+1} f(y) (z_{\rho(1)},...z_{\rho(k+1)}) D^0 g(y). \label{eq:leibnitz4}
\end{equation}
Isolating the value $i= 0$ in the sum involved in term $(b)$, we obtain
\begin{align}
(b) = & \sum_{i= 1}^k \sum_{\rho \in S_{i,k-i}} D^{i} f(y)(z_{\rho(1)},...,z_{\rho(i)})
D^{k-i+1} g(y) (z_{\rho(i+1)},..., z_{\rho(k)}, z_{k+1}) \notag \\
& \qquad + f(y) D^{k+1} g(y)(z_1,...,z_{k+1}). \label{eq:leibnitz5}
\end{align}
Observe that the last term of the r.h.s.\@ can be written as follows:
\begin{equation}
f(y) D^{k+1} g(y) (z_1,...,z_{k+1})
= \sum_{\rho \in S_{0,k+1}} D^0 f(y) D^{k+1} g(y) (z_{\rho(1)},...z_{\rho(k+1)}). \label{eq:leibnitz6}
\end{equation}
We can now combine \eqref{eq:leibnitz1}-\eqref{eq:leibnitz6}. In particular, the terms involved in the sums in \eqref{eq:leibnitz3} and \eqref{eq:leibnitz5} can be combined together thanks to the representation of $S_{i,k+1-i}$ provided in \eqref{eq:leibnitz1}.
We finally obtain
\begin{align*}
(a) + (b)= \ & \sum_{i=0}^{k+1} \sum_{\sigma \in S_{i,k+1-i}} D^i f(y)(z_{\sigma(1),...,\sigma(i)}) D^{k+1-i} g(y) (z_{\sigma(i+1)},...,z_{\sigma(k+1)}) \\
= \ & \sum_{i=0}^{k+1} \binom{k+1}{i} \text{Sym}_{i,k+1-i} \big( D^i f(y) \otimes D^{k+1-i}g(y) \big) (z_1,...,z_{k+1}).
\end{align*}
In the last inequality, we used that $|S_{i,k+1-i}|= \binom{k+1}{i}$.
The Leibnitz formula is proved for $k+1$. This concludes the proof.
\end{proof}

\begin{proof}[Proof of Lemma \ref{lemma:symmetry}]
The first part of the lemma follows directly from the definition and from the fact that $|S_{i,j}|= \binom{i+j}{i}$. Assume that $\mathcal{T}_1$ and $\mathcal{T}_2$ are symmetric.
Let us set $f\colon y \in Y \mapsto \mathcal{T}_1(y^{\otimes i})$ and $g\colon y \in Y \mapsto \mathcal{T}_2(y^{\otimes j})$. By Lemma \ref{lemma:continuousTensors2}, the functions $f$ and $g$ are both infinitely many times differentiable. Applying the Leibnitz formula to $fg$, we obtain
\begin{equation*}
D^{i+j}[f(0)g(0)]
= \sum_{\ell= 0}^{i+j} \binom{i+j}{\ell} \text{Sym}_{\ell,i+j-\ell} \big( D^\ell f(0) \otimes D^{i+j-\ell} g(0) \big).
\end{equation*}
The derivatives of $f$ of order $k > i$ are all null and the derivatives of $g$ of order $k > j$ are also all null. Therefore, in the above sum, all the terms vanish, except the one obtained for $\ell= i$. Moreover, since $\mathcal{T}_1$ and $\mathcal{T}_2$ are symmetric,
\begin{equation*}
D^i f(0)= i! \, \mathcal{T}_1 \quad \text{and} \quad D^j g(0)=  j! \, \mathcal{T}_2.
\end{equation*}
We therefore obtain that
\begin{equation*}
D^{i+j}[f(0)g(0)] = (i+j)! \, \text{Sym}_{i,j} \big( \mathcal{T}_1 \otimes \mathcal{T}_2 \big).
\end{equation*}
This proves that $\text{Sym}_{i,j} \big( \mathcal{T}_1 \otimes \mathcal{T}_2 \big)$ is a symmetric multilinear form, since it can be expressed as the $(i+j)$-th derivative of an infinitely many times differentiable function. The lemma is proved.
\end{proof}

\bibliographystyle{siam}

\end{document}